\documentclass[11pt,twoside]{amsart}
\usepackage{amssymb,amsmath}
\usepackage{amscd}
\usepackage{amsthm}
\usepackage{latexsym}
\usepackage[noadjust]{cite}
\usepackage{mathrsfs}
\usepackage{indentfirst}
\usepackage{graphicx}
\usepackage{graphicx}
\usepackage{subfigure}
\usepackage{color}
\usepackage{enumitem}
\allowdisplaybreaks[4]

\def\dfrac{\displaystyle\frac}

\numberwithin{equation}{section}

\newtheorem{theorem}{Theorem}[section]
\newtheorem{lemma}[theorem]{Lemma}
\newtheorem{proposition}[theorem]{Proposition}

\newtheorem{corollary}[theorem]{Corollary}

\usepackage[colorlinks=true,
linkcolor=blue,      
anchorcolor=red,  
citecolor=blue,        
linktocpage,pdfpagelabels,
bookmarksnumbered,bookmarksopen]{hyperref}
\usepackage[hyperpageref]{backref}

\makeatletter
\newcommand{\itemcustom}[2]{%
	\refstepcounter{enumi}%
	\item[#1]%
	\protected@edef\@currentlabel{#1}%
	\label{#2}%
}
\makeatother

\makeatletter
\providecommand\@dotsep{5}
\def\listtodoname{List of Todos}
\def\listoftodos{\@starttoc{tdo}\listtodoname}
\makeatother

\newcommand{\vp}{\varphi}

\newcommand{\eps}{\varepsilon}

\def\N{\mathbb{N}}
\def\R{\mathbb{R}}

\def\Qc{\mathcal{Q}}

\def\C{\mathcal{C}}

\def\RN{\mathbb{R}^N}

\def\n{\nabla}

\def\a{\alpha}

\def\l{\lambda}

\def\t{\theta}
\def\irn{\int_{\RN}}

\newcommand{\weakto}{\rightharpoonup}
\newcommand{\nequiv}{\not\equiv}

\newcommand{\cC}{{\mathcal C}}

\newcommand{\cF}{{\mathcal F}}

\newcommand{\cH}{{\mathcal H}}
\newcommand{\cI}{{\mathcal I}}

\newcommand{\cM}{{\mathcal M}}
\newcommand{\cN}{{\mathcal N}}

\newcommand{\cP}{{\mathcal P}}
\newcommand{\cQ}{{\mathcal Q}}

\title[Ground state solutions to Born-Infeld-Choquard problem]{Ground state solutions to Born-Infeld-Choquard problem}

\author[J. Mederski]{Jaros\l aw Mederski}
\author[X. Zeng]{
	Xiangjian Zeng}

\address{J. Mederski \hfill\break  
	\newline\indent Institute of Mathematics,
	\newline\indent Polish Academy of Sciences  
	\newline\indent ul. \'Sniadeckich 8, 00-656 Warsaw, Poland
	\newline\indent and
	\newline\indent 
	Faculty of Mathematics and Computer Science,		\newline\indent 
	Nicolas Copernicus University, \newline\indent ul. Chopina 12/18,		 87-100 Toruń, Poland
}
\email{\href{mailto:jmederski@impan.pl}{jmederski@impan.pl}}

\address{X. Zeng
	\hfill\break  \newline\indent
	\newline\indent Institute of Mathematics,
	\newline\indent Polish Academy of Sciences  
	\newline\indent ul. \'Sniadeckich 8, 00-656 Warsaw, Poland
	\newline\indent and
	\newline\indent Faculty of Mathematics, Informatics and Mechanics 
	\newline\indent University of Warsaw 
	\newline\indent ul. Banacha 2, 02-097 Warsaw, POLAND}
\email{\href{mailto:xzeng@impan.pl}{xzeng@impan.pl}}


\begin{document}
	\begin{abstract} 
		In this paper, we investigate the existence and qualitative properties of ground state solutions for the nonlocal Born-Infeld-Choquard problem 
		\begin{equation*}
			\begin{cases}
				-{\rm div}\left(\dfrac{\n u}{\sqrt{1-|\n u|^2}}\right)+ \omega u=\big(I_\alpha\ast |u|^{p}\big)|u|^{p-2}u, & \hbox{in }\RN,\; N\geq 3,
				\\[5mm]
				u(x)\to 0, &\hbox{as }|x|\to +\infty.
			\end{cases}
		\end{equation*}
		where $p>\frac{N+\a}{N}$, $\omega=0,1$ and $0<\a<N$.
		The equation is driven by the mean curvature operator in Lorentz-Minkowski space, motivated by the Born-Infeld nonlinear electromagnetic theory, and is coupled with a Choquard-type nonlocal nonlinearity. Due to the inherent relativistic gradient constraint $|\nabla u| \le 1$, the associated energy functional lacks standard $\mathcal{C}^1$ regularity, preventing the direct use of classical variational techniques.  We employ a non-smooth critical point theory on appropriate Pohožaev-type manifold to establish the existence of ground state solutions. Such non-smooth critical point theorem is abstract and we further show that it can be employed for strongly indefinite problem as well.  We also demonstrate that these solutions  are radially symmetric, and monotonously decay to zero at infinity. 
		
	\end{abstract}
	\maketitle

	\section{Introduction}

	To resolve the infinite energy problem of point charges in classical Maxwell theory, Born and Infeld \cite{Bnat,B,BInat,BI} introduced a nonlinear electrodynamics model. By inherently restricting the maximum field strength, their theory ensures that the electromagnetic field generated by a point charge has finite energy. Mathematically, this framework is governed by the Born-Infeld operator:
	\begin{equation*}
		\mathcal{Q}(u)=-{\rm div}\left(\frac{\nabla u}{\sqrt{1-|\nabla u|^2}}\right).
	\end{equation*}
	This operator naturally enforces the physical field bound via the mathematical condition $|\nabla u| < 1$. Furthermore, it holds profound geometric significance in classical relativity, where it represents the mean curvature operator in Lorentz-Minkowski space \cite{BS,CY}.
	In last years many authors focused their attention on problems related to  $\Qc$ in the whole $\RN$, with $N\ge 1$. In particular, some results for
	$\mathcal{Q}(u)=\rho$ in $\RN$
	can be found in \cite{BCF,BDP,BDPR,Bon-Iac2,Bon-Iac3,H,K,K-corr,DMT_AMS}, under different assumptions on $\rho$, where $\rho$ can be considered as a prescribed charge source.

	Much less is known, however, in the presence of a nonlinearity, and the case of a nonlocal nonlinearity seems to remain open. For equations of local type,
	\begin{equation}\label{cong}
		-{\rm div}\left(\dfrac{\nabla u}{\sqrt{1-|\nabla u|^2}}\right)=g(u), \qquad \hbox{in } \mathbb{R}^N,
	\end{equation}
	where $g$ is a nonlinear function of $u$, one of the first papers dealing with this kind of problem by variational methods is \cite{BDD}. In that work, Bonheure, Derlet and De Coster considered the case $g(s)=|s|^{p-2}s$, with $p>2^*=\frac{2N}{N-2}$ and $N\ge 3$. By means of suitable truncation arguments and ODE methods in the radial setting, they obtained the existence of finite-energy solutions. We also mention \cite{A,A2,P,BCP,MedPom,BalMedPom}, where \eqref{cong} was studied again by means of ODE techniques in the radial setting.
	
	Later, in \cite{BIMM}, Byeon et al. developed a monotonicity trick for nonsmooth functionals, inspired by the framework introduced in \cite{Szu}. They proved the existence of infinitely many radial solutions of \eqref{eq}, as well as  nonradial solutions when $N=4$ or $N\geq 6$, see also \cite{BIMM2024}. More recently, an approach not relying on symmetry assumptions was developed in \cite{BIM}, yielding the existence of a radial ground state solution and nonradial solutions for all $N\geq 4$.
	
	If the local nonlinearity in \eqref{cong} is replaced by a nonlocal one, the resulting problem becomes the so-called \emph{Born--Infeld--Choquard} problem, which reads as follows:
	\begin{equation}\label{eq}
		\begin{cases}
			-{\rm div}\left(\dfrac{\nabla u}{\sqrt{1-|\nabla u|^2}}\right)+ \omega u
			=\big(I_\alpha\ast |u|^{p}\big)|u|^{p-2}u, & \hbox{in } \mathbb{R}^N,\; N\geq 3,
			\\[5mm]
			u(x)\to 0, & \hbox{as } |x|\to +\infty,
		\end{cases}
	\end{equation}
	where $p>\frac{N+\alpha}{N}$, $0<\alpha<N$ and $I_\alpha(x):=\frac{C_{N,\alpha}}{|x|^{N-\alpha}}$ is the Riesz potential.
	Here, we assume that $\omega\in\{0,1\}$: the case $\omega=1$ is called the \emph{positive-mass case}, while the case $\omega=0$ is called the \emph{zero-mass case}.
	
	If the Born--Infeld operator in \eqref{eq} is replaced by the Laplacian, one obtains the well-known Choquard equation, also called the Choquard--Pekar equation in the case $N=3$, $\omega=1$, $\alpha=2$, and $p=2$:
	\begin{equation*}
		-\Delta u+u = \big(I_\alpha\ast |u|^p\big)|u|^{p-2}u\qquad \hbox{in } \mathbb{R}^N.
	\end{equation*}
	This equation arises, for instance, as an approximation to the Hartree--Fock theory of plasma \cite{Pekar,Lieb}. 
	A variational approach to this case was presented by Lieb \cite{Lieb} and Lions \cite{Lions}.
	More generally, for $N\geq 3$ and
	$
	\frac{N+\alpha}{N}<p<\frac{N+\alpha}{N-2},
	$
	Moroz and Van Schaftingen \cite{MVJFA,MVTAMS} investigated the existence, regularity, and qualitative properties of ground state solutions to the generalized Choquard problem, see also \cite{MVJFPTA} and references therein.
	
	Throughout this paper, we assume that $N\geq 3$ and $0<\alpha<N$. Unlike problem \eqref{cong}, however, problem \eqref{eq} cannot be treated in a straightforward way by using the variational methods available in the literature. More precisely, equation \eqref{eq} is, at least formally, the Euler--Lagrange equation associated with the energy functional
	\begin{equation*}
		\mathcal{I}(u):=
		\int_{\mathbb{R}^N}\left(1-\sqrt{1-|\nabla u|^2}\right)\,dx
		+\frac{\omega}{2}\int_{\mathbb{R}^N}|u|^2\,dx
		-\frac{1}{2p}\int_{\mathbb{R}^N}\big(I_\alpha\ast |u|^{p}\big)|u|^{p}\,dx.
	\end{equation*}
	However, the gradient part of this functional, namely
	\begin{equation*}
		\Psi_0(u):= 
		\int_{\mathbb{R}^N}\left(1-\sqrt{1-|\nabla u|^2}\right)\,dx,
	\end{equation*}
	is well defined only under the constraint $|\nabla u|\leq 1$ a.e. in $\mathbb{R}^N$. Consequently, the associated energy functional is not of class $\mathcal{C}^1$ on its natural domain.
	
	This lack of smoothness prevents the direct application of classical variational techniques. Indeed, the presence of the constraint $|\nabla u|\leq 1$ a.e. in $\mathbb{R}^N$, together with the singular behavior of the integrand on the set where $|\nabla u|=1$, requires different and nonstandard strategies.
	
	Moreover, the  presence non-local nonlinearity makes the variational framework more delicate than in the local case e.g. \cite{BIMM,BIM}.
	
	We define the functional space $X_\omega$ by 
	\begin{equation*}
		X_\omega:=\overline{\cC_{0}^{\infty}(\RN)}^{\lVert\cdot\rVert_{X_\omega}},\quad\text{where }\lVert u\rVert_{X_\omega}:=\sqrt{\omega\lVert u\rVert_{2}^2+\lVert \nabla u\rVert_{2}^2+\lVert \nabla u\rVert_{r}^2}.
	\end{equation*}
	with $r>\max\{2^*,N\}$ fixed. Here we denote the $L^{p}$-norm by $\lVert\cdot\rVert_{p}$ and $2^*=\frac{2N}{N-2}$. In view of the singularity, we define the {\em effective domain} of $\Psi_0$, and hence of $\mathcal{I}$, as 
	\begin{equation*}
		D(\Psi_0):= \{u\in X_\omega: \Psi_0(u)<\infty\}=\{u\in X_\omega: |\n u| \leq 1\text{ a.e. on }\RN\}.
	\end{equation*}
	Although $\mathcal{I}$ depends on $\omega$, we omit the notation $\mathcal{I}_\omega$ whenever this does not lead to confusion.
	
	Our aim is to investigate the {\em ground state solutions} to problem \eqref{eq}, that is, the nontrivial least energy solutions to \eqref{eq}. To make this notion precise, we need a preliminary illustration. 
	Let
	\begin{equation}\label{nonlocal functional}
		\cF_{\a,p}(u):= \irn\big(I_\alpha\ast |u|^{p}\big)|u|^{p}  \,dx,\quad\hbox{ for }u\in X_\omega.
	\end{equation}
	Then we introduce a {\em Pohožaev functional} related to our problem:
	\begin{equation}\label{Pohozaev functional}
		\begin{aligned}
			\cP_\omega(u):&=  N\int_{\mathbb{R}^N} 1 - \sqrt{1 - |\nabla u|^2}\, dx + \frac{N+2}{2}\omega\irn |u|^2\, dx - \frac{N+\alpha+2p}{2p} \irn\big(I_\alpha\ast |u|^{p}\big)|u|^{p}  \,dx\\
			& =N\Psi_{0}(u) +\frac{N+2}{2}\omega\lVert u\rVert_{2}^2 - \frac{N+\alpha+2p}{2p}\cF_{\a,p}(u)
		\end{aligned}
	\end{equation}
	and the {\em Pohožaev set}
	\begin{equation*}
		\cM_\omega:=\{ u\in D(\Psi_{0})\backslash\{0\} : \cP_\omega(u)=0 \}.
	\end{equation*}
	Therefore, we now define the {\em ground state level } by
	\begin{equation*}
		c_\omega:= \inf_{u\in \cM_\omega} \cI (u).
	\end{equation*}
	If $c_\omega$ can be achieved by some $u_\omega$ in $\cM_\omega$, then we will show in Theorems \ref{Thm 1} and \ref{thm 0mass&sobolev ineq} that it is indeed a solution to \eqref{eq} and will be called a {\em ground state solution} to \eqref{eq}. As we shall see, a ground state solution has the least energy among all nontrivial solutions of \eqref{eq}, which belong to $\mathcal{M}_\omega$; see Lemma \ref{Pohozaev lemma} below.
	
	Here, we emphasize that, although the functional $\mathcal{I}$ is not well defined on the whole space $X_\omega$, it can still be treated within the nonsmooth variational framework due to Szulkin \cite{Szu}; see also \cite{BIMM}. In this setting, critical points in the convex sense are shown to be weak solutions of \eqref{eq}. It follows from \cite[Proposition 2.10]{BIMM} for the linear problem involving $\cQ$.
	
	More precisely, let $\Psi:X_\omega\to (-\infty,+\infty]$ be a lower semicontinuous convex functional and let $\Phi\in \mathcal{C}^1(X_\omega,\mathbb{R})$. We say that $u\in X_\omega$ is a {\em critical point} of $\Psi-\Phi$ if
	\begin{equation}\label{critical point}
		\Psi(v)-\Psi(u)-\Phi'(u)(v-u) \geq 0
		\qquad \text{for all } v\in X_\omega.
	\end{equation}
	In the case of \eqref{eq}, we take
	\begin{equation*}
		\Phi(u):= \frac{1}{2p}\mathcal{F}_{\alpha,p}(u) - \frac{\omega}{2}\lVert u\rVert_{2}^{2},
		\qquad
		\Psi(u):=\Psi_0(u).
	\end{equation*}
	We shall see that the functional $\mathcal{F}_{\alpha,p}$, and hence $\Phi$, is of class $\mathcal{C}^1$ on the underlying functional space, see Proposition \ref{differentiable of nonlocal}.
	
	Our first main result in the positive mass case is stated as follows.
	\begin{theorem}\label{Thm 1}
		Let $p>\frac{N+\alpha}{N}$. Then every minimizing sequence $(u_n)_{n\in \mathbb{N}}$ for $c_1$ admits, up to translations, a subsequence that converges strongly to some $u_1\in \mathcal{M}_1$. Moreover, $u_1$ is a  ground state solution to \eqref{eq} with $\omega=1$.
	\end{theorem}
	
	Next, we show the qualitative properties of the ground state solutions to \eqref{eq}.
	
	\begin{theorem}\label{thm of qualitative}
		Let $p>\frac{N+\a}{N}$. If $u$ is a ground state solution to \eqref{eq} with $\omega=1$, then either $u>0$ or $u<0$ on $\RN$. Moreover, it is radially symmetric and decays monotonically to zero, i.e., there is a single-variable function $v:(0,+\infty)\to\R$ such that $u(x)= v(|x-x_0|)$ for some $x_0\in \RN$ and $v(r)\to 0$ monotonously as $r\to +\infty$.
	\end{theorem}
	
	We are also interested in the nonlocal version for the Sobolev-type inequality, which is a counterpart to the local one that has been studied in \cite{BIM,BDD}. Specifically, the following type of inequality 
	\begin{equation}\label{Sobolev-type inequality}
		\irn (1-\sqrt{1-|\n u|^2})\, dx \geq C_{p,\alpha}\left(\irn (I_{\a}*|u|^{p})|u|^{p}\, dx\right)^{\frac{N}{N+\a+2p}},
	\end{equation}
	will be investigated. More precisely, the best constant $C_{p,\alpha}$ can be characterized by the minimizer over a suitable Pohožaev constraint $\cM_0$, which one will see that it in fact can be attained. Moreover, such minimizer is a ground state solution to \eqref{eq} with $\omega=0$,
	where $p>\frac{N+\a}{N-2}$.

	Our last main result concerns the zero mass case.
	\begin{theorem}\label{thm 0mass&sobolev ineq}
		Let  $p>\frac{N+\alpha}{N-2}$. Then every minimizing sequence $(u_n)_{n\in \N}$ for $c_0$ has, up to translation, a convergent subsequence converging strongly to some $u_0$ in $\cM_0$, which is a classical ground state solution to \eqref{eq} with $\omega=0$. Moreover, it holds that
		\begin{enumerate}[label={\rm (\roman*)}, ref={\rm (\roman*)}]
			\item for any $u\in X_0$ the inequality \eqref{Sobolev-type inequality} holds,
			where $$C_{p,\alpha}= \frac{(N+\a+2p)c_{0}^{\frac{N+2p}{N+\a+2p}}}{(2Np)^{\frac{N}{N+\a+2p}}(\a+2p)^{\frac{\a+2p}{N+\a+2p}}}.$$
			\item  every optimizer $v$ has a form that $v(x)= tu(\frac{x}{t})$ with $t>0$, where $u$ is a minimizer for $c_0$. 
		\end{enumerate}
		
	\end{theorem}
	
	We would like to mention that if $\alpha=0$, then $C_{p,\alpha}$ coincides with the constant obtained in the local version of \eqref{Sobolev-type inequality} in \cite{BIM}.

	We can also say more about the quantitative properties with such minimizers.
	
	\begin{theorem}\label{thm:zeroradial}
		Let $p>\frac{N+\alpha}{N-2}$. Then any minimizer for $c_0$ or optimizer for \eqref{Sobolev-type inequality} has constant sign, that is, it is either positive or negative on whole $\RN$. Moreover, it is radially symmetric and decays monotonically to zero.
	\end{theorem}
	
	We emphasize that both Theorems \ref{thm of qualitative} and \ref{thm:zeroradial} remain valid in the local case $\alpha=0$. In particular, they yield the radial symmetry of all ground state solutions, a property which was not established in \cite{BIM}.
	
	The paper is organized as follows. In Section \ref{sec preliminaries}, we collect the basic notation, recall the functional framework, and establish some preliminary properties of the nonlocal term and of the nonsmooth Born--Infeld functional. In particular, we prove the differentiability of the Choquard-type functional and derive a Poho\v{z}aev identity for weak solutions. Section \ref{sec CCP and PD} is devoted to the concentration-compactness analysis. We prove a vanishing lemma, a nonlocal Br\'ezis--Lieb type decomposition, and a profile decomposition adapted to the nonlocal setting. In Section \ref{sec Positive mass case}, we study the positive-mass case \(\omega=1\). We analyze the fibering map associated with the natural scaling, prove the existence of a minimizer on the Poho\v{z}aev set \(\mathcal M_1\), and show that such a minimizer is a ground state solution. Section \ref{sec Qualitative properties} contains the qualitative analysis of ground state solutions, including their constant sign, radial symmetry, and monotonicity properties. In Section \ref{sec zero mass}, we turn to the zero-mass case \(\omega=0\). We prove the existence of a ground state solution, establish the associated nonlocal Sobolev-type inequality, and characterize the best constant and its optimizers. Finally, in the Appendix \ref{Appendix}, we prove an abstract result, Lemma \ref{abstract critical point}, showing that constrained minimizers on a suitable constrained set are critical points. The argument is inspired by \cite{BIM}, but its formulation is adapted to both local and nonlocal problems. We also demonstrate that such abstract lemma can be employed for strongly indefinite Schr\"{o}dinger equations of Szulkin and Weth's settings \cite{SW} as well, see Theorem \ref{ThNehariPankov}.  We expect that it will be useful in a broader class of nonsmooth variational settings.

	.
	
	\section{Preliminaries}\label{sec preliminaries}
	
	We write $A\lesssim B$ if there exists a constant $C>0$, independent of the relevant parameters, such that $A\leq CB$.
	
	We first recall from \cite{BIMM} the basic embedding properties of our working space.
	
	\begin{proposition}\label{embedding}
		Let $r>\max\{N,2^*\}$. Then $X_1$ and $X_0$ are reflexive Banach space, $X_1\subset X_0$ and satisfy the following continuous embeddings:
		\begin{enumerate}[label={\rm (\roman*)}, ref={\rm (\roman*)}]
			\item\label{embedding of W1r} $X_0\hookrightarrow W^{1,q}(\RN)$, for every $q\in [2^*,r]$; and $X_1\hookrightarrow W^{1,q}(\RN)$, for every $q\in [2,r]$
			\item\label{embedding of C0} $X_1\text{ and }X_0\hookrightarrow \cC_{0}(\RN) : = \{u\in \cC(\RN):\lim\limits_{|x|\to \infty}|u(x)| =0 \}$;
			\item\label{embedding of Lp} $X_1\hookrightarrow L^{q}(\RN)$,  for every $q\in [2,\infty]$; and $X_0\hookrightarrow L^{q}(\RN)$,  for every  $q\in [2^*,\infty]$
		\end{enumerate}
	\end{proposition}

	Next, we show that the nonlocal part of the functional is of class $\cC^1$. 
	\begin{proposition}\label{differentiable of nonlocal}
		Let  $p\geq\frac{N+\alpha}{N}$ in the case of $X_1$ and $p\geq \frac{N+\a}{N-2}$ in the case of $X_0$. Then $I_{\a}\ast |u|^{p}\in L^{\infty}(\RN)$ for every $u\in X_\omega$. Moreover, the functional $\cF_{\a,p}(u)$ defined in \eqref{nonlocal functional} is of class $\cC^1$.
	\end{proposition}
	\begin{proof}
		In view of Proposition \ref{embedding} \ref{embedding of Lp}, we have $|u|^{p}\in L^{r}(\R^N)$ for any $r\in [\frac{2N}{N+\a},\infty]$. Since $\frac{2N}{N+\a}<\frac{N}{\a}$, the standard property of Riesz's potential (see \cite{DuP,LiebLoss}) leads to $I_{\a}\ast |u|^{p}\in L^{s}(\RN)$, where $s\in [\frac{2N}{N-\a},\infty)$. Moreover, since $|u|^{p}\in L^\infty(\mathbb{R}^N)\cap L^{\frac{2N}{N+\alpha}}(\mathbb{R}^N)$, we obtain that 
		\begin{equation}\label{eq:Linfty}
			I_\alpha*|u|^{p}\in L^\infty(\mathbb{R}^N).
		\end{equation}
		Indeed, setting $f:=|u|^p$, for every $x\in\mathbb{R}^N$, we write
		\[
		|I_\alpha*f(x)|
		\leq
		C_{N,\alpha}
		\left(
		\int_{B_1(x)}
		\frac{|f(y)|}{|x-y|^{N-\alpha}}\,dy
		+
		\int_{\mathbb{R}^N\setminus B_1(x)}
		\frac{|f(y)|}{|x-y|^{N-\alpha}}\,dy
		\right).
		\]
		For the local part, using $f\in L^\infty(\mathbb{R}^N)$, we have
		\[
		\int_{B_1(x)}
		\frac{|f(y)|}{|x-y|^{N-\alpha}}\,dy
		\leq
		\|f\|_\infty
		\int_{B_1(x)}
		\frac{1}{|x-y|^{N-\alpha}}\,dy\lesssim \|f\|_\infty.
		\]
		For the far-field part, we apply Hölder's inequality with
		$r=\frac{2N}{N+\alpha}$,
		$r'=\frac{2N}{N-\alpha}$, namely
		\[
		\begin{aligned}
			\int_{\mathbb{R}^N\setminus B_1(x)}
			\frac{|f(y)|}{|x-y|^{N-\alpha}}\,dy
			&\leq
			\|f\|_{\frac{2N}{N+\alpha}}
			\left(
			\int_{\mathbb{R}^N\setminus B_1(x)}
			\frac{1}{|x-y|^{(N-\alpha)r'}}\,dy
			\right)^{1/r'} 	\lesssim\|f\|_{\frac{2N}{N+\alpha}}.
		\end{aligned}
		\]
		Therefore,
		\[
		\|I_\alpha*f\|_\infty
		\lesssim
		\|f\|_\infty+
		\|f\|_{\frac{2N}{N+\alpha}}.
		\]
		and the proof of \eqref{eq:Linfty} is complete.

		It remains to verify the differentiability. Put
		\[
		s:=\frac{2N}{N+\a}.
		\]
		By Proposition \ref{embedding} \ref{embedding of Lp}, the map $u\mapsto |u|^p$ sends $X_\omega$ continuously into $L^s(\RN)$. Indeed, if $u_n\to u$ in $X_\omega$, then $u_n\to u$ in $L^{ps}(\RN)$, and
		\[
		\big\||u_n|^p-|u|^p\big\|_s
		\leq C\big(\lVert u_n\rVert_{ps}^{p-1}+\lVert u\rVert_{ps}^{p-1}\big)
		\lVert u_n-u\rVert_{ps}\to0.
		\]
		Hence, by the Hardy-Littlewood-Sobolev inequality,
		\[
		\cF_{\a,p}(u)=\int_{\RN}(I_\a*|u|^p)|u|^p\,dx
		\]
		is continuous on $X_\omega$. Moreover, since the Nemytskii map
		$u\mapsto |u|^{p-2}u$ is continuous from $X_\omega$ into
		$L^{\frac{2N}{(N+\a)(p-1)}}(\RN)$, for any $v\in X_\omega$ we have
		\begin{equation}\label{derivative of nonlocal}
			\cF_{\a,p}'(u)v
			=2p\int_{\RN}(I_\a*|u|^p)|u|^{p-2}uv\,dx.
		\end{equation}
		The continuity of $u\mapsto \cF_{\a,p}'(u)$ follows again from the above
		continuity of $u\mapsto |u|^p$, the Hardy-Littlewood-Sobolev inequality
		and Hölder's inequality. Therefore $\cF_{\a,p}\in \cC^1(X_\omega,\R)$.
	\end{proof}
	
	We next prove the Pohožaev-type identity for \eqref{eq}.
	
	\begin{lemma}\label{Pohozaev lemma}
		Any weak solution to 
		\begin{equation*}
			-{\rm div}\left(\dfrac{\n u}{\sqrt{1-|\n u|^2}}\right)+ \omega u=\big(I_\alpha\ast |u|^{p}\big)|u|^{p-2}u,\quad  u\in X_{\omega};\; N\geq 3,
		\end{equation*}
		where $p\geq\frac{N+\alpha}{N}$, if $\omega=1$; or $p\geq\frac{N+\a}{N-2}$,  if $\omega=0$, satisfies the following formula 
		\begin{equation}\label{Pohozaev}
			\int_{\mathbb{R}^N} \frac{|\nabla u|^2}{\sqrt{1 - |\nabla u|^2}}\,dx 
			=N\int_{\mathbb{R}^N}\left( 1 - \sqrt{1 - |\nabla u|^2}+\frac{\omega}{2}|u|^2\right)\,dx -  \frac{N+\alpha}{2p} \irn\big(I_\alpha\ast |u|^{p}\big)|u|^{p}  \,dx.
		\end{equation}
		Therefore, any  nontrivial weak solution to \eqref{eq} lies in $\cM_\omega$.
	\end{lemma}
	\begin{proof}
		Let $u$ be a weak solution to the equation in this Lemma.  Note from Proposition \ref{embedding} that $u\in L^{\infty}(\RN)$ and from Proposition \ref{differentiable of nonlocal} that $I_\a* |u|^{p}\in L^{\infty}(\RN)$. Then by \cite[Proposition 2.10]{BIMM}, we know that \( u \in W^{2,q}_{\mathrm{loc}}(\mathbb{R}^N) \) for each \( q \in [2, \infty) \), and \( |\nabla u| < 1 \) on \( \mathbb{R}^N \). Let $\eta\in \cC_{0}^{\infty}(\RN)$ such that $\eta\equiv 1$ when $|x|\leq 1$, $\eta\equiv 0$ when $|x|\geq 2$, and $|x| \cdot |\nabla \eta(x)|\leq C$ for every $x\in \RN$. For any $\varepsilon>0$, set $\eta_{\varepsilon}(x): =\eta(\varepsilon x)$. By regularity of $u$, we know that $\vp_{\varepsilon}:=\eta_{\varepsilon}(x)x\cdot \nabla u \in W^{1,q}_{\rm{loc}}(\RN)$ for each $q\in [2,\infty)$. Testing \eqref{eq} by $\vp_{\varepsilon}$ leads to
		\begin{equation*}
			\irn \frac{\nabla u\cdot\nabla \vp_{\varepsilon}}{\sqrt{1-|\nabla u|^2}} \, dx=\irn  \big(I_{\a} \ast |u|^{p}\big)\big(|u|^{p-2}u\vp_{\varepsilon}\big) \, dx-\omega\irn u\vp_{\varepsilon}\, dx.
		\end{equation*}  
		
		Notice that
		\[
		\frac{\nabla u \cdot \nabla \vp_\varepsilon}{\sqrt{1 - |\nabla u|^2}}
		= \frac{(\nabla u\cdot \nabla \eta_{\eps})(\nabla u \cdot x)}{\sqrt{1 - |\nabla u|^2}}
		+ \frac{\eta_\varepsilon |\nabla u|^2}{\sqrt{1 - |\nabla u|^2}}
		+ \eta_\varepsilon x \cdot \nabla \left( 1 - \sqrt{1 - |\nabla u|^2} \right).
		\]
		Integrating by part leads to
		\begin{equation*}
			\begin{aligned}
				\irn \frac{\nabla u\cdot\nabla \vp_{\varepsilon}}{\sqrt{1-|\nabla u|^2}} \, dx=&\irn \big[\frac{(\nabla u\cdot \nabla \eta_{\eps})(\nabla u \cdot x)}{\sqrt{1 - |\nabla u|^2}}+ \frac{\eta_\varepsilon |\nabla u|^2}{\sqrt{1 - |\nabla u|^2}}\big]dx
				\\
				&- \irn\big( x \cdot \nabla \eta_\varepsilon+N \eta_\eps \big) \left( 1 - \sqrt{1 - |\nabla u|^2} \right)dx 
			\end{aligned}
		\end{equation*}
		Since 
		\[
		\rho:=(I_\alpha*|u|^p)|u|^{p-2}u-\omega u
		\in L^\infty(\mathbb R^N)
		\]
		and $\mathcal{Q}(u)=\rho$,
		by \cite[Proposition 2.10(ii)]{BIMM}, there exists
		\(\delta\in(0,1)\) such that
		\[
		\|\nabla u\|_\infty\leq 1-\delta.
		\]
		Consequently,
		\[
		\frac{|\nabla u|^2}{\sqrt{1-|\nabla u|^2}}
		\leq C_\delta |\nabla u|^2\in L^1(\mathbb R^N).
		\]
		Therefore, using Lebesgue's dominated convergence theorem we get 
		\begin{equation}\label{eq:LebPoho}
			\lim_{\eps\to 0^+}\irn \frac{\nabla u\cdot\nabla \vp_{\varepsilon}}{\sqrt{1-|\nabla u|^2}} \, dx= \irn \frac{|\nabla u|^2}{\sqrt{1 - |\nabla u|^2}}\, dx-N\irn \left( 1 - \sqrt{1 - |\nabla u|^2} \right)dx .
		\end{equation}
		Next, it suffices to follow the proof in \cite[Proposition 3.1]{MVJFA} to obtain 
		\[
		\lim_{\eps\to 0^+} \irn \big(I_{\a} \ast |u|^{p}\big)\big(|u|^{p-2}u\vp_{\varepsilon}\big) \, dx=-\frac{N+\a}{2p}\irn \big(I_{\a}\ast |u|^{p}\big)|u|^{p}\, dx
		\]		
		and
		\begin{equation*}
			\lim_{\eps\to 0^+}\omega\irn u\vp_{\varepsilon}\, dx =-\omega\frac{N}{2} \irn |u|^2 \, dx,
		\end{equation*}
		where, in the case $\omega=0$, this term is absent.
		Then \eqref{Pohozaev} holds. Moreover, testing \eqref{eq}  by $u$ and combining it with Pohožaev identity \eqref{Pohozaev}, it is not difficult to deduce $\cP_\omega(u)=0$.   
	\end{proof}

	\begin{corollary}\label{cor:nonexistence}
		Assume that either
		\[
		\omega=1
		\quad\text{and}\quad
		1<p\leq \frac{N+\alpha}{N},
		\]
		or
		\[
		\omega=0
		\quad\text{and}\quad
		1<p\leq \frac{N+\alpha}{N-2}.
		\]
		Let \(u\in X_\omega\) be a weak solution to \eqref{eq} such that
		\begin{equation}\label{eq:uinL1}
			\frac{|\nabla u|^2}{\sqrt{1-|\nabla u|^2}}\in L^1(\mathbb R^N).
		\end{equation}
		Then \(u=0\).
	\end{corollary}
	\begin{proof}
		Suppose that \(u\in X_\omega\) is a weak solution to \eqref{eq}. By the regularity result, see for instance \cite[Proposition 2.10]{BIMM}, and by a standard approximation argument, we may test \eqref{eq} by \(u\). This gives
		\begin{equation}\label{Nehari}
			\int_{\mathbb R^N}
			\frac{|\nabla u|^2}{\sqrt{1-|\nabla u|^2}}\,dx
			+
			\omega\int_{\mathbb R^N}|u|^2\,dx
			=
			\int_{\mathbb R^N}
			\big(I_\alpha * |u|^p\big)|u|^p\,dx.
		\end{equation}
		Thanks to the additional assumption \eqref{eq:uinL1},
		the proof of Lemma \ref{Pohozaev lemma} can be repeated in the present range of \(p\). Hence the Poho\v{z}aev identity holds, namely
		\begin{equation}\label{Pohozaev critical exp}
			\int_{\mathbb R^N}
			\frac{|\nabla u|^2}{\sqrt{1-|\nabla u|^2}}\,dx
			=
			N\int_{\mathbb R^N}
			\left(
			1-\sqrt{1-|\nabla u|^2}
			+\frac{\omega}{2}|u|^2
			\right)\,dx
			-
			\frac{N+\alpha}{2p}
			\int_{\mathbb R^N}
			\big(I_\alpha * |u|^p\big)|u|^p\,dx .
		\end{equation}
		Combining \eqref{Nehari} and \eqref{Pohozaev critical exp}, we obtain
		\begin{equation}\label{eq:nonexistence-combined}
			\begin{aligned}
				0
				&=
				\frac{N+\alpha+2p}{2p}
				\int_{\mathbb R^N}
				\frac{|\nabla u|^2}{\sqrt{1-|\nabla u|^2}}\,dx
				-
				N\int_{\mathbb R^N}
				\left(1-\sqrt{1-|\nabla u|^2}\right)\,dx  \\
				&\quad
				+
				\frac{N+\alpha-Np}{2p}
				\omega
				\int_{\mathbb R^N}|u|^2\,dx .
			\end{aligned}
		\end{equation}
		If \(\omega=1\), then \(p\leq \frac{N+\alpha}{N}\), and therefore
		\[
		\frac{N+\alpha-Np}{2p}\geq0,
		\qquad
		\frac{N+\alpha+2p}{2p}\geq\frac N2.
		\]
		Thus, from \eqref{eq:nonexistence-combined},
		\[
		0
		\geq
		\frac N2
		\int_{\mathbb R^N}
		\left[
		\frac{|\nabla u|^2}{\sqrt{1-|\nabla u|^2}}
		-
		2\left(1-\sqrt{1-|\nabla u|^2}\right)
		\right]\,dx.
		\]
		If \(\omega=0\), then \(p\leq \frac{N+\alpha}{N-2}\), and hence
		\[
		\frac{N+\alpha+2p}{2p}\geq\frac N2.
		\]
		Again, \eqref{eq:nonexistence-combined} yields the same inequality:
		\[
		0
		\geq
		\frac N2
		\int_{\mathbb R^N}
		\left[
		\frac{|\nabla u|^2}{\sqrt{1-|\nabla u|^2}}
		-
		2\left(1-\sqrt{1-|\nabla u|^2}\right)
		\right]\,dx.
		\]
		For \(t\in[0,1)\), we have
		\[
		\frac{t}{\sqrt{1-t}}
		-
		2(1-\sqrt{1-t})
		=
		\frac{(1-\sqrt{1-t})^2}{\sqrt{1-t}}
		\geq0.
		\]
		Taking \(t=|\nabla u|^2\), we infer that \(\nabla u=0\) a.e. in \(\mathbb R^N\). Since \(u\in X_\omega\hookrightarrow C_0(\mathbb R^N)\), it follows that \(u\equiv0\).
	\end{proof}
	
	Note that, in particular, Corollary \ref{cor:nonexistence} shows the nonexistence of nontrivial solutions in the Hardy--Littlewood--Sobolev critical cases, namely when
	\[
	p=\frac{N+\alpha}{N}
	\quad\text{if } \omega=1,
	\qquad
	p=\frac{N+\alpha}{N-2}
	\quad\text{if } \omega=0.
	\]	
	Indeed, in this case, condition \eqref{eq:uinL1} is automatically satisfied; see Lemma \ref{Pohozaev lemma}.
	
	\section{Concentration-compactness and profile decomposition}\label{sec CCP and PD}

	Throughout this section, we assume that
	\[
	p>\frac{N+\alpha}{N}
	\quad\text{if } \omega=1,
	\qquad
	p>\frac{N+\alpha}{N-2}
	\quad\text{if } \omega=0.
	\]

	\begin{lemma}\label{vanishing lemma}
		Suppose that $(u_n)_{n\in\N}\subset X_\omega$ is bounded and there exists $r>0$ such that
		\begin{equation*}
			\limsup_{n\to\infty}\sup_{y\in\RN} \int_{B(y,r)} |u_n|^q\, dx =0,
		\end{equation*}
		where 
		\begin{equation}\label{exp in two cases}
			q=\begin{cases}
				2\qquad\text{if }\omega=1,\\
				2^*\qquad \text{if }\omega=0.
			\end{cases}
		\end{equation}
		Then it holds that
		\begin{equation*}
			\lim_{n\to \infty}\irn (I_{\a}* |u_n|^p) |u_n|^{p}\, dx=0.
		\end{equation*}
	\end{lemma}
	\begin{proof}
		In view of Proposition \ref{embedding}, the sequence $(u_n)$ is bounded in $L^s(\mathbb R^N)$ for every $s\geq2$ if $\omega=1$, and for every $s\geq2^*$ if $\omega=0$. Note that $\frac{2Np}{N+\a}>q$. Then, arguing as in \cite[Lemma 1.21]{Willem},  one has 
		\begin{equation}\label{fine estimate for nonlocal Lp}
			\irn |u_n|^{\frac{2Np}{N+\a}}\leq C\left(\sup_{y\in\RN} \int_{B(y,r)} |u_n|^q\, dx\right)^{1-\t} \lVert u_n\rVert_{X_\omega}^{\t s}
		\end{equation}
		for some fixed $s>\frac{2Np}{N+\a}$, some positive constant $C$ and $\t$.
		Thus, Hardy-Littlewood-Sobolev inequality gives
		\begin{equation*}
			\begin{aligned}
				\irn (I_{\a}* |u_n|^p) |u_n|^{p}\, dx&\leq C\left(\irn |u_n|^{\frac{2Np}{N+\a}}\right)^{\frac{N+\a}{N}}\\
				&\leq C\left(\sup_{y\in\RN} \int_{B(y,r)} |u_n|^q\, dx\right)^{\frac{(1-\t)(N+\a)}{N}} \lVert u_n\rVert_{X_\omega}^{\frac{(N+\a)\t s}{N}}\to 0
			\end{aligned}
		\end{equation*}
		as $n\to \infty$.
	\end{proof}
	
	The following is a variant version of the 
	Br\'{e}zis-Lieb lemma \cite{BL}.
	\begin{lemma}\label{Brezis-Lieb}
		Suppose that $(u_n)$ is bounded in $X_\omega$ and $u_n\to u_0$ a.e. on $\RN$. Then we have
		\begin{equation*}
			\lim_{n\to\infty}\left(\irn (I_{\a}* |u_n|^p) |u_n|^{p}\, dx - \irn (I_{\a}* |u_n-u_0|^p)|u_n-u_0|^{p}\, dx \right) =\irn (I_{\a}*|u_0|^{p})|u_0|^{p}\, dx 
		\end{equation*}
	\end{lemma}
	\begin{proof}
		From Proposition \ref{embedding} \ref{embedding of Lp} and the assumptions, we know that $(|u_n|^{p})$ is bounded in $L^{\frac{2N}{N+\a}}(\RN)$.
		Note that, for every $n\in \N$,
		\begin{equation*}
			\begin{aligned}
				\irn &(I_{\a}* |u_n|^p) |u_n|^{p}\, dx - \irn (I_{\a}* |u_n-u_0|^p)|u_n-u_0|^{p}\, dx \\
				&= \irn (I_{\a}*( |u_n|^p- |u_n-u_0|^{p}))( |u_n|^{p}- |u_n-u_0|^{p})\, dx \\
				& + 2 \irn (I_{\a}* (|u_n|^{p}-|u_n-u_0|^p))|u_n-u_0|^{p}\, dx.
			\end{aligned}
		\end{equation*}
		By classical Br\'{e}zis-Lieb Lemma (see, for example, \cite[Lemma 1.32]{Willem}), one has $|u_n|^p- |u_n-u_0|^{p}\to |u_0|^{p}$ strongly in $L^{\frac{2N}{N+\a}}(\RN)$ as $n\to\infty$. Then the Hardy-Littlewood-Sobolev inequality implies that $I_{\a}*( |u_n|^p- |u_n-u_0|^{p})\to I_{\a}*|u_0|^{p}$ strongly in $L^{\frac{2N}{N-\a}}(\RN)$ as $n\to\infty$. Since \(u_n-u_0\to0\) a.e. on $\RN$ and \((u_n-u_0)\) is bounded in
		\(L^{\frac{2Np}{N+\alpha}}(\R^N)\), we have
		\[
		|u_n-u_0|^p\rightharpoonup0
		\quad\text{weakly in }
		L^{\frac{2N}{N+\alpha}}(\R^N).
		\] as $n\to\infty$, which reaches the conclusion.
	\end{proof}
	
	Define that 
	\begin{equation}\label{auxillary Psi1}
		\Psi_1(u)=\Psi_0(u)+\omega\frac{\lVert u\rVert_{2}^{2}}{2}=\int_{\mathbb{R}^N}\left( 1 - \sqrt{1 - |\nabla u|^2}\right)\, dx+\omega\frac{\lVert u\rVert_{2}^{2}}{2}.
	\end{equation}
	Next we show the profile decomposition for the functional.
	
	\begin{lemma}\label{profile decomposition}
		Suppose that $(u_n)_{n\in \N}$ is a bounded sequence in $X_\omega$. Then there exists $k\in \N\cup \{+\infty\}$, $(y_{n}^{i})_{n\in\N}\subset\RN$ and $\tilde{u}_{i}\in X_\omega$ for $0\leq i < k+1$\footnote{To avoid ambiguity, we make the convention that if $k=+\infty$, then $k+1=+\infty$.} such that
		\begin{equation*}
			\begin{aligned}
				u_{n}(\cdot+y_{n}^{i})&\weakto \tilde{u}_{i}\qquad\text{{\rm weakly in }}X_\omega,\, \text{{\rm as }}n\to\infty; \\
				\lim_{n\to \infty}|y_{n}^{i}-y_{n}^{j}|&= +\infty\qquad\text{{\rm for }}i\neq j;\\
				\liminf_{n\to\infty}\Psi_1(u_n)&\geq \sum_{i=0}^{k}\Psi_1(\tilde{u}_{i})\quad\left(\liminf_{n\to\infty}\Psi_0(u_n)\geq \sum_{i=0}^{k}\Psi_0(\tilde{u}_{i})\right);\\
				\lim_{n\to \infty}\cF_{\a,p}(u_n) &= \sum_{i=0}^{k}\cF_{\a,p}(\tilde{u}_{i}),
			\end{aligned}
		\end{equation*}
		where $\cF_{\a,p}$ is defined by \eqref{nonlocal functional}.
	\end{lemma}
	\begin{proof}
		Suppose that, up to a subsequence, $u_n\weakto \tilde{u}_{0}$ weakly in $X_\omega$ and $u_n(x)\to \tilde{u}_{0}(x)$ a.e. on $\RN$. Let $v_{n}^{0}:=u_n-\tilde{u}_0$ and 
		\begin{equation*}
			\delta_{0}:= \limsup_{n\to\infty}\sup_{y\in\RN}\int_{B(y,1)} |v_{n}^{0}|^q\, dx,
		\end{equation*}
		where $q$ satisfies \eqref{exp in two cases}. When $\delta_{0}=0$, we then deduce by Lemma \ref{vanishing lemma} and \ref{Brezis-Lieb}  that $\cF_{\a,p}(u_n)\to \cF_{\a,p}(\tilde{u}_0)$ as $n\to \infty$. Moreover, observe that the functionals $\Psi_1$ and $\Psi_0$ are convex and lower semicontinuous, it holds that $\liminf\limits_{n\to\infty}\Psi_1(u_n)\geq\Psi_1(\tilde{u}_0)$ and $\liminf\limits_{n\to\infty}\Psi_0(u_n)\geq\Psi_0(\tilde{u}_0)$. Thus, it is sufficient to take $k=0$ and $y_{n}^{0}=0$ for all $n\in \N$ to obtain the assertion. 
		
		Otherwise, if $\delta_{0}>0$, one may find a sequence $(y_{n}^{1})_{n\in\N}\subset\RN$ satisfying $|y_{n}^{1}|\to+\infty$ such that 
		\[
		\int_{B(y_{n}^{1},1)}|v_{n}^{0}|^{q}\, dx\to \delta_0, \qquad\text{as }n\to \infty.
		\]
		Hence, we may suppose that, up to a subsequence, $v_{n}^{0}(\cdot+y_{n}^{1})\weakto \tilde{u}_{1}\not\equiv 0$ weakly in $X_\omega$ and $v_{n}^{0}(x+y_{n}^{1})\to \tilde{u}_{1}(x)$ a.e. on $\RN$. Since $v_{n}^{0}= u_{n}-\tilde{u}_0$,  we obtain  that $u_{n}(\cdot+y_{n}^{1})\weakto \tilde{u}_{1}$ weakly in $X_\omega$. Now we set $v_{n}^{1}:=v_{n}^{0}-\tilde{u}_1(\cdot-y_{n}^{1})=u_{n}-\tilde{u}_0-\tilde{u}_1(\cdot-y_{n}^{1})$ and 
		\begin{equation*}
			\delta_{1}:= \limsup_{n\to\infty}\sup_{y\in\RN}\int_{B(y,1)} |v_{n}^{1}|^q\, dx.
		\end{equation*}
		If $\delta_{1}=0$, we can conclude by letting $k=1$, $(y_{n}^{0})$, $(y_{n}^{1})$ as above. Indeed, note that Lemma \ref{Brezis-Lieb} gives
		\begin{equation*}
			\cF_{\a,p}(u_n)=\cF_{\a,p}(\tilde{u}_0)+\cF_{\a,p}(v_{n}^{0})+o(1)=\cF_{\a,p}(\tilde{u}_0)+\cF_{\a,p}(v_{n}^{1})+\cF_{\a,p}(\tilde{u}_{1})+o(1).
		\end{equation*}
		Again by  Lemma \ref{vanishing lemma}  we have $\cF_{\a,p}(u_{n})\to \cF_{\a,p}(\tilde{u}_1)+\cF_{\a,p}(\tilde{u}_0)$ as $n\to \infty$, which gives the assertion of nonlocal term. For $\Psi_1$, we first write, for any domain $\Omega\subset\RN$,
		\begin{equation*}
			\Psi_{\Omega} (u):=\int_{\Omega} \left(1 - \sqrt{1 - |\nabla u|^2}+\omega\frac{|u|^2}{2}\right)\, dx.
		\end{equation*}
		Then, it holds by the fact $|y_{n}^{1}-y_{n}^{0}|=|y_{n}^{1}|\to \infty$ as $n\to \infty$, and Fatou's lemma that
		\begin{equation*}
			\begin{aligned}
				\liminf_{n\to\infty} \Psi_1(u_n) \geq \liminf_{n\to\infty}\sum_{i=0}^{1} \Psi_{B(y_{n}^{i},R)}(u_n)=\liminf_{n\to\infty}\sum_{i=0}^{1} \Psi_{B(0,R)}(u_n(\cdot+y_{n}^{i}))\geq \sum_{i=0}^{1} \Psi_{B(0,R)}(\tilde{u}_{i})
			\end{aligned}
		\end{equation*}
		for any $R>0$. Therefore, we conclude it by letting $R\to\infty$. In addition the same reasoning is still valid for $\Psi_0$.
		
		On the other hand, when $\delta_1>0$, as above we can find another sequence $(y_{n}^{2})_{n\in\N}\subset\RN$ satisfying $|y_{n}^{2}|\to+\infty$ such that 
		\[
		\int_{B(y_{n}^{2},1)}|v_{n}^{1}|^q\, dx\to \delta_1, \qquad\text{as }n\to \infty.
		\]
		Furthermore, we assume that, up to a subsequence,  $v_{n}^{1}(\cdot+y_{n}^{2})\weakto \tilde{u}_{2}\not\equiv 0$ weakly in $X_\omega$ as well as $v_{n}^{1}(x+y_{n}^{2})\to \tilde{u}_{2}(x)$ a.e. on $\RN$, that is, $v_{n}^{0}(\cdot+y_{n}^{2})-\tilde{u}_{1}(\cdot+y_{n}^{2}-y_{n}^{1})\weakto \tilde{u}_{2}$ weakly in $X_\omega$ as $n\to\infty$. Since $\tilde{u}_{2}\nequiv0$, we get that $|y_{n}^{2}-y_{n}^{1}|\to\infty$ and  $u_{n}(\cdot+y_{n}^{2})\weakto \tilde{u}_2$ weakly in $X_\omega$ as $n\to \infty$.  Now, let $v_{n}^{2}:= v_{n}^{1}-\tilde{u}_{2}(\cdot-y_{n}^{2})$ and 
		\begin{equation*}
			\delta_{2}:= \limsup_{n\to\infty}\sup_{y\in\RN}\int_{B(y,1)} |v_{n}^{2}|^q\, dx.
		\end{equation*}
		If $\delta_{2}=0$, we set $k=2$ and thus prove the results as before. If not, that is, $\delta_{2}>0$, then we once more find $(y_{n}^{3})_{n\in\N}\subset \RN$ satisfying $|y_{n}^{3}|\to+\infty$ such that
		\[
		\int_{B(y_{n}^{3},1)}|v_{n}^{2}|^q\, dx\to \delta_2, \qquad\text{as }n\to \infty,
		\] 
		and we repeat the same procedure. In case that there exists some finite $k\in \N$ such that $\delta_{k}=0$, then the almost same discussion yields the conclusions.
		
		We now assume that $k=+\infty$. Based on the construction above, for any $i\in \N$ we have found sequence $(y_{n}^{i})_{n\in\N}\subset\RN$ and $\tilde{u}_{i}\in X_\omega\backslash\{0\}$ such that $u_{n}(\cdot+y_{n}^{i})\weakto \tilde{u}_{i}$ weakly in $X_\omega$ as $n\to\infty$ and $|y_{n}^{i}-y_{n}^{j}|\to \infty$ with $i\neq j$. Using the same notation above, for every finite $m\in \N$ and $R>0$, we exploit Fatou's lemma to obtain 
		\begin{equation}\label{lower semi estimate}
			\begin{aligned}
				\liminf_{n\to\infty} \Psi_1(u_n) \geq \liminf_{n\to\infty}\sum_{i=0}^{m} \Psi_{B(y_{n}^{i},R)}(u_n)=\liminf_{n\to\infty}\sum_{i=0}^{m} \Psi_{B(0,R)}(u_n(\cdot+y_{n}^{i}))\geq \sum_{i=0}^{m} \Psi_{B(0,R)}(\tilde{u}_{i}).
			\end{aligned}
		\end{equation}
		Hence the conclusion for $\Psi_1$ holds by letting $R\to\infty$ and $m\to \infty$.  To show the assertion for $\cF_{\a,p}$, first observe from the construction that  
		\begin{equation*}
			\begin{aligned}
				\int_{B(0,1)} |\tilde{u}_{i+1}|^q\, dx=\delta_{i}> \delta_{i+1}=\int_{B(0,1)} |\tilde{u}_{i+2}|^q\, dx, \qquad\text{for }i\in\N.
			\end{aligned}
		\end{equation*}
		We claim that $\delta_{j}\to 0$ as $j\to +\infty$. In fact, by classical Br\'{e}zis-Lieb lemma (e.g. \cite[Lemma 1.32]{Willem}), one has 
		\begin{equation*}
			\lVert u_n\rVert_{q}^q=\sum_{i=0}^{m}\lVert \tilde{u}_{i}\rVert_{q}^{q}+\lVert v_{n}^{m}\rVert_{q}^q+o(1)
		\end{equation*}
		for every $m\in\N$. Therefore, by Proposition \ref{embedding} \ref{embedding of Lp} we have 
		\begin{equation*}
			0<\sum_{i=0}^{m-1}\delta_{i} =\sum_{i=0}^{m-1}\int_{B(0,1)} |\tilde{u}_{i+1}|^q\, dx \leq  \lVert u_n\rVert_{q}^q-\lVert v_{n}^{m}\rVert_{q}^q-\lVert \tilde{u}_{0}\rVert_{q}^q+o(1)\leq \sup_{n\in\N} \lVert u_n\rVert_{X}<+\infty,
		\end{equation*}
		which implies that $\sum_{i=0}^{\infty}\delta_{i}<+\infty$ and $\delta_{j}\to 0$ as $j\to\infty$. 
		
		Next, our goal is to show that sequences $(v_{n}^{m})_{n\in\N}$ are bounded uniformly with respect to every $m\in\N$. To this aim, first note that for every $m\in\N$, 
		\begin{equation*}
			v_{n}^{m}(x)=u_n(x)-\sum_{i=0}^{m}\tilde{u}_{i}(x-y_{n}^{i}),\qquad |y_{n}^{i}-y_{n}^{j}|\to+\infty, \text{ as }n\to\infty.
		\end{equation*}
		It follows that
		\begin{equation*}
			\lim_{n\to\infty}\left\lVert \sum_{i=0}^{m}\tilde{u}_{i}(\cdot-y_{n}^{i})\right\rVert_{X}=\sum_{i=0}^{m}\lVert \tilde{u}_{i}\rVert_{X}.
		\end{equation*}
		On the other hand, from the weak lower semicontinuity of norm, similar reasoning of \eqref{lower semi estimate} yields
		\begin{equation*}
			\sum_{i=0}^{m}\lVert \tilde{u}_{i}\rVert_{X}\leq \liminf_{n\to\infty}\lVert u_n\rVert_{X}
		\end{equation*}
		for every $m\in\N$. Thus, 
		\begin{equation*}
			\limsup_{n\to\infty}\lVert v_{n}^{m}\rVert_{X}\leq \limsup_{n\to\infty}\lVert u_{n}\rVert_{X}+\limsup_{n\to\infty}\left\lVert \sum_{i=0}^{m}\tilde{u}_{i}(\cdot-y_{n}^{i})\right\rVert_{X}\leq 2\limsup_{n\to\infty}\lVert u_{n}\rVert_{X}<C,
		\end{equation*}
		where $C>0$ is independent from $m$. Then, from \eqref{fine estimate for nonlocal Lp} and Hardy-Littlewood-Sobolev inequality we have 
		\begin{equation*}
			\begin{aligned}
				\cF_{\a,p}(v_{n}^{m})&\leq C \left(\irn |v_{n}^{m}|^{\frac{2Np}{N+\a}}\right)^{\frac{N+\a}{N}}\\
				&\leq C\left(\sup_{y\in\RN} \int_{B(y,1)} |v_{n}^{m}|^q\, dx\right)^{\frac{(1-\t)(N+\a)}{N}} \lVert v_{n}^{m}\rVert_{X}^{\frac{(N+\a)\t s}{N}}\\
				&\leq C\delta_{m}^{\frac{(1-\t)(N+\a)}{N}}\to 0\qquad\text{as }m\to\infty.
			\end{aligned}
		\end{equation*}
		Thus, by Lemma \ref{Brezis-Lieb} we have 
		\begin{equation*}
			\cF_{\a,p}(u_n)= \sum_{i=0}^{m}\cF_{\a,p}(\tilde{u}_{i})+\cF_{\a,p}(v_{n}^{m})+o(1)
		\end{equation*}
		and get the desire by letting $m\to+\infty$. 
	\end{proof}
	
	\section{Existence result for positive mass case}\label{sec Positive mass case}
	
	This section is devoted to proving the existence of the ground state solution for $\omega=1$, which is the constrained minimizer on the Pohožaev set $\cM_1$. Throughout this section we always assume that $p>\frac{N+\a}{N}$. We first introduce the following scaling for any $u\in X_1$:
	\begin{equation}\label{scaling}
		u_{t}(x):= tu\left(\frac{x}{t}\right) \qquad\text{for }x\in\R^N.\; t\in (0,+\infty).
	\end{equation}
	It is easy to see that, for any $u\in D(\Psi_0)\backslash\{0\}$ and $t\in (0,+\infty)$, $u_{t}\in D(\Psi_0)\backslash\{0\}$. Moreover, we have
	\begin{equation}\label{fiber function}
		\begin{aligned}
			\cI(u_{t})&= t^N\int_{\mathbb{R}^N} 1 - \sqrt{1 - |\nabla u|^2}\, dx+\frac{t^{N+2}}{2}\lVert u\rVert_{2}^{2}-\frac{t^{N+\a+2p}}{2p}\cF_{\a,p}(u),\\
			\frac{d\cI(u_{t})}{dt}&=t^{N-1}\left(N\int_{\mathbb{R}^N} 1 - \sqrt{1 - |\nabla u|^2}\, dx+\frac{N+2}{2}t^{2}\lVert u\rVert_{2}^{2}-\frac{N+\a+2p}{2p}t^{\a+2p}\cF_{\a,p}(u)\right).
		\end{aligned}
	\end{equation}
	We observe that 
	\begin{equation*}
		\frac{d\cI(u_{t})}{dt}\Big\rvert_{t=1}=\cP_1(u)\qquad\text{for }u\in D(\Psi_0)\backslash\{0\},
	\end{equation*}
	where $\cP_1$ is defined by \eqref{Pohozaev functional}. Next, we investigate the structure of functional under such scaling.
	\begin{lemma}\label{uniqueness of parametre on Pohozaev}
		For each $u\in D(\Psi_0)\backslash\{0\}$, there exists a unique $t_{u}\in (0,+\infty)$, which is a maximum of function $t\mapsto\cI(u_t)$, such that $\cP_1(u_{t_{u}})=0$, i.e., $u_{t_{u}}\in \cM_1$. Therefore, it holds that $\cP_1(u_t)>0$ for $0<t<t_u$ and $\cP_1(u_t)<0$ for $t>t_u$. Moreover, the map $u\mapsto t_{u}$ is continuous on $D(\Psi_0)\backslash\{0\}$.
	\end{lemma}
	\begin{proof}
		For any $u\in D(\Psi_{0})\backslash\{0\}$, set 
		\begin{equation*}
			\cH(t,u):= N\Psi_0(u)+\frac{N+2}{2}t^{2}\lVert u\rVert_{2}^{2}-\frac{N+\a+2p}{2p}t^{\a+2p}\cF_{\a,p}(u).
		\end{equation*}
		Then 
		\begin{equation*}
			\frac{d\cH(t,u)}{dt}=(N+2)t\lVert u\rVert_{2}^{2}-\frac{(N+\a+2p)(\a+2p)}{2p}t^{\a+2p-1}\cF_{\a,p}(u).
		\end{equation*}
		Since $\a+2p-1>1$, we get 
		\begin{equation*}
			\lim\limits_{t\to 0}\cH(t,u)=N\Psi_{0}(u)\qquad\text{and}\qquad\lim\limits_{t\to +\infty}\cH(t,u)=-\infty.
		\end{equation*}
		Moreover, there exists a unique $T>0$ such that $d\cH(t,u)/dt>0$ as $t\in (0,T)$; and $d\cH(t,u)/dt<0$ as $t\in (T,+\infty)$.  Thus, it is obvious that there exists a unique $t_{u}>T$ such that $\cH(t_{u},u)=0$, so that $\cP_1(u_{t_{u}})=0$ and $u_{t_{u}}\in \cM_1$. Noticing that $d\cH(t,u)/dt|_{t=t_{u}}< 0$, then we exploit implicit function theorem on $\cH(t,u)$ to deduce the continuity.
	\end{proof}
	
	According to above Lemma, we know that the set $\cM_1$ is not empty and $c_1$ is well-defined. Next we show that the constrained minimizer on $\cM_1$ implies a critical point of $\cI$.
	\begin{lemma}\label{minimizer=critical point}
		Let $u_1\in \cM_1$ satisfy $\cI(u_1)=c_1$. Then $u_1$ is a critical point of $\cI$.
	\end{lemma}
	\begin{proof}
		We apply Lemma \ref{abstract critical point}. Let $X=X_1$,
		\[
		\mathfrak B(u):=\frac{1}{2p}\cF_{\a,p}(u)-\frac12\lVert u\rVert_2^2,
		\qquad
		\mathfrak A(u):=\Psi_0(u),
		\]
		so that $\cI=\mathfrak A-\mathfrak B$. By Proposition
		\ref{differentiable of nonlocal}, $\mathfrak B\in \cC^1(X_1,\R)$, while
		$\mathfrak A$ is proper, convex and lower semicontinuous on $X_1$, and
		$D(\mathfrak A)=D(\Psi_0)$ is convex. Let $S_tu=u_t$ be the scaling
		in \eqref{scaling}. Then $S_1u=u$, $S_tu\nequiv0$ for $u\nequiv0$, and
		$S_t(S_ru)=S_{tr}u$ for all $t,r>0$. Moreover,
		\[
		\mathfrak A(S_tu)=\Psi_0(u_t)=t^N\Psi_0(u)=t^N\mathfrak A(u),
		\]
		so that {\rm (S1)--(S3)} in Lemma \ref{abstract critical point} hold
		with $\alpha(t)=t^N$. The $\cC^2$-regularity of the fiber maps follows from
		\eqref{fiber function}.
		
		We next identify the abstract functional $\mathfrak R$. Since
		\[
		\mathfrak R(t,u):=\frac{d}{dt}\mathfrak B(u_t)
		=\frac{N+\a+2p}{2p}t^{N+\a+2p-1}\cF_{\a,p}(u)
		-\frac{N+2}{2}t^{N+1}\lVert u\rVert_2^2,
		\]
		we have $\mathfrak R\in\cC^1((0,+\infty)\times X_1,\R)$. Moreover,
		\[
		\mathfrak P(u)=\left.\frac{d}{dt}\cI(u_t)\right|_{t=1}= N\mathfrak A(u)-\mathfrak R(1,u)=\cP_1(u),
		\qquad
		\mathfrak M=\cM_1.
		\]
		Thus {\rm (A1)} and {\rm (A2)} follow directly from Lemma
		\ref{uniqueness of parametre on Pohozaev}. It remains to check {\rm (A3)}
		at the minimizer $u_1$. Since $u_1\in\cM_1$, we have
		\[
		\frac{N+\a+2p}{2p}\cF_{\a,p}(u_1)
		=N\Psi_0(u_1)+\frac{N+2}{2}\lVert u_1\rVert_2^2.
		\]
		Consequently,
		\[
		\begin{aligned}
			\left.\frac{d}{dt}\left(\frac{\mathfrak R(t,u_1)}{\alpha'(t)}\right)\right|_{t=1}
			&=\frac{(N+\a+2p)(\a+2p)}{2Np}\cF_{\a,p}(u_1)
			-\frac{N+2}{N}\lVert u_1\rVert_2^2\\
			&=(\a+2p)\Psi_0(u_1)
			+\frac{(N+2)(\a+2p-2)}{2N}\lVert u_1\rVert_2^2>0,
		\end{aligned}
		\]
		because $u_1\nequiv0$ and $p>\frac{N+\a}{N}$ implies $\a+2p-2>0$. Hence
		{\rm (A3)} is satisfied.
		
		All assumptions of Lemma \ref{abstract critical point} are therefore
		satisfied, and hence $u_1$ is a critical point of $\cI$.
	\end{proof}
	
	Next, we are going to show the ground state level $c_1$ is positive by which identify it with mountain pass level.
	
	\begin{lemma}\label{positive energy level}
		Let $m_1$ be the mountain pass level defined by
		\begin{equation}\label{mountain pass level}
			m_1:=\inf_{\gamma\in \Gamma} \max_{\tau \in [0,1]} \cI(\gamma(\tau)),\qquad  \Gamma := \{\gamma\in \C\left([0,1]; X_1\right): \gamma(0)=0,\,  \cI(\gamma(1))<0\}.
		\end{equation}
		Then it holds that $c_1= m_1>0$.
	\end{lemma}
	\begin{proof}
		From \eqref{fiber function} there holds that $\cI(u_t)\to -\infty$ as $t\to +\infty$ and $\cI(u_t)\to 0$ as $t\to 0$, for any nonzero $u\in D(\Psi_{0})$. Thus, the class $\Gamma$ is nonempty and $m_1$ is well-defined.
		
		We first prove $m_1>0$. By Hardy-Littlewood-Sobolev inequality and Proposition \ref{embedding} \ref{embedding of Lp}, one has 
		\begin{equation*}
			\cF_{\a,p}(u)\leq  C \left(\irn |u|^{\frac{2Np}{N+\a}}\right)^{\frac{N+\a}{N}}\leq C\lVert u\rVert_{X_1}^{2p}\qquad \text{for } u\in X_1.
		\end{equation*}
		If $u\notin D(\Psi_{0})$, it is natural to make the convention that $+\infty=\cI(u)>0$. Hence we only suppose that $u\in D(\Psi_{0})$, that is, $|\n u|\leq 1$ a.e. on $\RN$, which implies that $\lVert \n u\rVert_{r}^r\leq \lVert \n u\rVert_{2}^2$ for $r> \max\{N, 2^*\}$. Note that 
		\begin{equation}\label{equivalent of Psi0}
			\frac{1}{2}\lVert \n u \rVert_{2}^2\leq \Psi_0(u)\leq \lVert \n u \rVert_{2}^2\qquad \text{for all }u\in D(\Psi_{0})\backslash\{0\}.
		\end{equation} 
		Therefore, we have
		\begin{equation}\label{lower estimate of energy}
			\begin{aligned}
				\cI(u) &= \Psi_0(u)+\frac{1}{2}\lVert u\rVert_{2}^2- \frac{1}{2p}\cF_{\a,p}(u)\\
				&\geq \frac{1}{4}\left(\lVert \n u\rVert_{2}^2 + \lVert \n u\rVert_{r}^r\right) +\frac{1}{2}\lVert u\rVert_{2}^2 - C\lVert u\rVert_{X_1}^{2p}\\
				& \geq \frac{1}{4}\lVert u\rVert_{X_1}^{2} - C\lVert u\rVert_{X_1}^{2p}.
			\end{aligned}
		\end{equation} 
		Since $2p>2$, it thereby holds for sufficiently small $\beta>0$ that $\inf_{\lVert u\rVert_{X_1}=\beta}\cI(u)>0$. Therefore, $m_1>0$.
		
		Subsequently, we aim to show $c_1\geq m_1$.  For any $u\in \cM_1$, fix a $T_u>1$ such that $\cI(u_{T_u})<0$. Then we define the path $\gamma:[0,1]\mapsto X_1$ by $\gamma(\tau):= u_{\tau/T_u}=(\tau/T_u)u(T_u\cdot/\tau)$. Clearly, $\gamma\in \Gamma$ and $m_1\leq \max_{\tau \in [0,1]} \cI(\gamma(\tau)) = \cI(u)$ by Lemma \ref{uniqueness of parametre on Pohozaev}. According to arbitrary $u\in \cM_1$, we therefore obtain that 
		\begin{equation*}
			m_1\leq \inf_{u\in \cM_1}\cI(u)= c_1.
		\end{equation*}
		
		Finally, we focus on the proof for $c_1\leq m_1$, and it is sufficient to verify that $\gamma([0,1])\cup \cM_1\neq \emptyset$ for all $\gamma\in\Gamma$. First deduce analogously to \eqref{lower estimate of energy} that
		\begin{equation*}
			\begin{aligned}
				\cP_1(u)&= N\Psi_{0}(u) +\frac{N+2}{2}\lVert u\rVert_{2}^2 - \frac{N+\alpha+2p}{2p}\cF_{\a,p}(u)\\
				&\geq \frac{N}{4} \lVert u\rVert_{X_1}^{2} - C\lVert u\rVert_{X_1}^{2p}.
			\end{aligned}
		\end{equation*}
		Hence we get $\cP_1(\gamma(t))>0$ for sufficiently small $t$. On the other hand, one has
		\begin{equation*}
			\begin{aligned}
				\cP_1(\gamma(1))&= N\cI(\gamma(1)) + 2\left(\frac{\lVert \gamma(1)\rVert_{2}^2}{2}-\frac{\a+2p}{4p}\cF_{\a,p}(\gamma(1))\right)\\
				& \leq (N+2)\cI(\gamma(1)) -2\Psi_{0}(\gamma(1))<0.
			\end{aligned}
		\end{equation*}
		Therefore, by intermediate value theorem there exists some $t_0\in (0,1)$ such that $\cP_1(\gamma(t_0))=0$, that is, $\gamma(t_0)\in \cM_1$. Thus, we complete the proof.
	\end{proof}
	
	Now we are in the position to prove our first main result.
	\begin{proof}[Proof of Theorem \ref{Thm 1}]
		Given a minimizing sequence $(u_n)_{n\in\N}\subset \cM_1$ for $c_1$, from Lemma \ref{positive energy level} we have 
		\begin{equation*}
			\cI(u_n) = \frac{\a+2p}{N+\a+2p}\Psi_{0}(u_n)+\frac{\a+2p-2}{2(N+\a+2p)}\lVert u_n\rVert_{2}^2\to c_1>0
		\end{equation*}
		as $n\to\infty$, which implies that $\{\Psi_{0}(u_n)\}_{n\in\N}$ and $\{\lVert u_n\rVert_{2}\}_{n\in \N}$ are both bounded. Since all $(u_n)_{n\in\N}$ are in $D(\Psi_{0})$, then $\lVert \n u_n\rVert_{r}^r\leq \lVert \n u_n \rVert_{2}^2$ for all $n\in\N$. Therefore, we get from \eqref{equivalent of Psi0} that sequence $(u_n)_{n\in\N}$ is bounded in $X_1$. By Lemma \ref{profile decomposition}, there exists $k\in \N\cup \{+\infty\}$, $(y_{n}^{i})_{n\in\N}$ and $\tilde{u}_{i}\in X_1$ such that 
		\begin{equation}\label{splliting formulae}
			\begin{aligned}
				u_{n}(\cdot+y_{n}^{i})&\weakto \tilde{u}_{i}\qquad\text{{\rm weakly in }}X_1,\, \text{{\rm as }}n\to\infty; \\
				u_{n}(x+y_{n}^{i})&\to \tilde{u}_{i}(x)\qquad\text{{\rm a.e. on }}\RN,\, \text{{\rm as }}n\to\infty;\\ 
				\lim_{n\to \infty}|y_{n}^{i}-y_{n}^{j}|&= +\infty\qquad\text{{\rm for }}i\neq j;\\
				\liminf_{n\to\infty}\Psi(u_n)&\geq \sum_{i=0}^{k}\Psi(\tilde{u}_{i})\quad\left(\liminf_{n\to\infty}\Psi_0(u_n)\geq \sum_{i=0}^{k}\Psi_{0}(\tilde{u}_{i})\right);\\
				\lim_{n\to \infty}\cF_{\a,p}(u_n) &= \sum_{i=0}^{k}\cF_{\a,p}(\tilde{u}_{i}),
			\end{aligned}
		\end{equation}
		If $k=0$ and $\tilde{u}_0=0$, then 
		\begin{equation*}
			\cI(u_n) = \frac{\a+2p}{N+\a+2p}\Psi_{0}(u_n)+\frac{\a+2p-2}{2(N+\a+2p)}\lVert u_n\rVert_{2}^2\to 0
		\end{equation*}
		and $c_1=0$, which gives a contradiction. Hence, either $k>0$ or $k=0$ with $\tilde{u}_0\neq0$. Noticing that $D(\Psi_{0})$ is convex and closed, then it is weakly closed and all $\tilde{u}_{i}$ are located in $D(\Psi_{0})$.
		
		Next, we set $J:=\{i: \tilde{u}_i\nequiv 0\}$. The above assumption implies that $J\neq \emptyset$. By \eqref{splliting formulae} and Br\'{e}zis-Lieb Lemma, the Pohožaev identity $\cP_1(u_n)=0$ yields  
		\begin{equation*}
			\begin{aligned}
				\sum_{i\in J}N\Psi_{0}(\tilde{u}_i)\leq \liminf_{n\to \infty} N\Psi_{0}(u_n)&=\liminf_{n\to \infty}\left[ \frac{N+\a+2p}{2p}\cF_{\a,p}(u_n)-\frac{N+2}{2}\lVert u_n\rVert_{2}^2\right]\\
				&=\frac{N+\a+2p}{2p}\sum_{i\in J}\cF_{\a,p}(\tilde{u}_i)-\frac{N+2}{2}\sum_{i\in J}\lVert \tilde{u}_i\rVert_{2}^2.
			\end{aligned}
		\end{equation*}
		Hence, There exists $j\in J$ such that $\cP_1(\tilde{u}_j)\leq 0$. From Lemma \ref{uniqueness of parametre on Pohozaev} the parameter $t_j$ such that $t_{j}\tilde{u}_{j}(\frac{\cdot}{t_{j}}):= \tilde{w}_j\in \cM_1$ satisfies $t_j\leq 1$. On the other hand, it follows by \eqref{splliting formulae} and  Fatou's Lemma that
		\begin{equation}\label{energy obligatory}
			\begin{aligned}
				c_1\leq \cI(\tilde{w}_j)&=\frac{\a+2p}{N+\a+2p}\Psi_{0}(\tilde{w}_j)+\frac{\a+2p-2}{2(N+\a+2p)}\lVert \tilde{w}_j\rVert_{2}^2\\
				&=\frac{\a+2p}{N+\a+2p}t_{j}^{N}\Psi_{0}(\tilde{u}_j)+\frac{\a+2p-2}{2(N+\a+2p)}t_{j}^{N+2}\lVert \tilde{u}_j\rVert_{2}^2\\
				&\leq \frac{\a+2p}{N+\a+2p}\Psi_{0}(\tilde{u}_j)+\frac{\a+2p-2}{2(N+\a+2p)}\lVert \tilde{u}_j\rVert_{2}^2\\
				&\leq \liminf_{n\to \infty}\left[\frac{\a+2p}{N+\a+2p}\Psi_{0}(u_{n}(\cdot+y_{n}^{j}))+\frac{\a+2p-2}{2(N+\a+2p)}\lVert u_n(\cdot+y_{n}^{j})\rVert_{2}^2\right]\\
				&=\liminf_{n\to \infty}\cI(u_n)= c_1.
			\end{aligned}
		\end{equation}
		Therefore, we conclude that $t_j=1$ and $\tilde{u}_j\in \cM_1$ with $\cI(\tilde{u}_j)=c_1$, which indicates that $\tilde{u}_j$ is a minimizer for $c_1$. Moreover, from the last two formulae in \eqref{splliting formulae} and \eqref{energy obligatory} it is clear that no other nontrivial profile can occur. For convenience, we relabel $\tilde{u}_j$ by $u_1$ and set $y_n:=y_n^j$.
		
		We next prove the strong convergence. From \eqref{energy obligatory} and the convergence $\cI(u_n)\to c_1$, we have
		\[
		\Psi_{0}(u_n(\cdot+y_n))\to \Psi_{0}(u_1)
		\qquad\text{and}\qquad
		\lVert u_n(\cdot+y_n)\rVert_{2}\to \lVert u_1\rVert_{2}.
		\]
		Since $u_n(\cdot+y_n)\weakto u_1$ in $X_1$, the convergence in $X_1$
		follows from \cite[Lemma 2.2]{BIM}. Thus every minimizing sequence
		admits, up to translations, a strongly convergent subsequence whose
		limit belongs to $\cM_1$ and achieves $c_1$.
		
		It remains to show that this minimizer is a classical ground state solution. By Lemma \ref{minimizer=critical point}, $u_1$ is a critical point of $\cI$. Equivalently,
		\[
		\Psi_0(v)-\Psi_0(u_1)+\langle u_1,v-u_1\rangle_{L^2(\RN)}
		-\int_{\RN}(I_{\a}*|u_1|^p)|u_1|^{p-2}u_1(v-u_1)\,dx\geq0
		\]
		for all $v\in X_1$. Hence $u_1$ is a minimizer of
		\[
		I_{\rho}(\psi):= \irn (1-\sqrt{1-|\n \psi|^2})\, dx -\langle \rho, \psi\rangle_{L^2(\RN)}
		\]
		with
		\[
		\rho=(I_{\a}*|u_1|^{p})|u_1|^{p-2}u_1-u_1.
		\]
		By Proposition \ref{embedding} \ref{embedding of Lp} and Proposition
		\ref{differentiable of nonlocal}, we have $u_1,I_{\a}*|u_1|^p\in
		L^\infty(\RN)$. Therefore $\rho\in L^\infty(\RN)$, and
		\cite[Proposition 2.10]{BIMM} gives that $u_1$ is a weak solution of
		\[
		-{\rm div}\left(\frac{\nabla u_1}{\sqrt{1-|\nabla u_1|^2}}\right)+u_1
		=(I_\alpha*|u_1|^p)|u_1|^{p-2}u_1\qquad\text{in }\RN.
		\]
		
		Finally, we aim at the regularity of the solution. Arguing as in the proof of Lemma \ref{Pohozaev lemma}, one has $u_1\in W^{2,q}_{\mathrm{loc}}(\RN)$ for each $q\in[2,\infty)$ and $|\nabla u_1|<1$ on $\RN$. Hence $u_1\in \cC_{\mathrm{loc}}^{1,\lambda}$ for some $0<\lambda\leq1$ and the coefficients
		\[
		a_{ij}(x)=\frac{\delta_{ij}}{\sqrt{1-|\n u_1|^2}}+\frac{\partial_{i}u_1 \partial_{j}u_1}{\left(1-|\n u_1|^2\right)^{3/2}}
		\]
		belong to $\cC^{0,\lambda}_{\mathrm{loc}}(\RN)$. On the other hand, since $|u_1|^{p}\in L^{s}(\RN)$ for all $s\in[\frac{2N}{N+\a},\infty]$, \cite[Theorem 2]{DuP} yields $I_{\a}*|u_1|^{p}\in \cC_{\mathrm{loc}}^{0,\lambda}(\RN)$. Thus $(I_{\a}*|u_1|^{p})|u_1|^{p-2}u_1\in \cC_{\mathrm{loc}}^{0,\lambda}(\RN)$, and Schauder's regularity theory applied to
		\[
		\sum_{i,j=1}^{N}a_{ij}(x)\partial_{ij}u_1
		=(I_{\a}*|u_1|^{p})|u_1|^{p-2}u_1-u_1
		\]
		implies that $u_1\in \cC^2(\RN)$. Since $u_1\in\cM_1$, it is nontrivial. Moreover, if $w$ is any nontrivial weak solution of \eqref{eq} with $\omega=1$, then Lemma \ref{Pohozaev lemma} gives $w\in\cM_1$, and hence $\cI(w)\geq c_1=\cI(u_1)$. Therefore $u_1$ has the least energy among all nontrivial solutions, and so it is a classical ground state solution of \eqref{eq}.
	\end{proof}
	
	\section{Qualitative properties of ground state solutions}\label{sec Qualitative properties}
	In this section we show that all the ground state solutions to \eqref{eq} must have constant sign and are radially symmetric. 
	\begin{proposition}\label{constant sign}
		Let $p>\frac{N+\a}{N}$. If $u\in \cM_1$ is a ground state solution to \eqref{eq}, then either $u>0$ or $u<0$ on $\RN$. 
	\end{proposition}
	\begin{proof}
		Let $u\in \cM_1$ be a ground state solution to \eqref{eq}. Then $\cI(u)=c_1$. Since
		\begin{equation*}
			\Psi_0(|u|)=\Psi_0(u),\qquad \lVert |u|\rVert_2=\lVert u\rVert_2,\qquad
			\cF_{\a,p}(|u|)=\cF_{\a,p}(u),
		\end{equation*}
		we have $|u|\in \cM_1$ and $\cI(|u|)=c_1$. Hence, by Lemma \ref{minimizer=critical point}, $|u|$ is also a critical point of $\cI$. Arguing as in the proof of Theorem \ref{Thm 1}, $|u|$ is a classical nonnegative solution to \eqref{eq}. By the strong maximum principle, $|u|>0$ on $\RN$. Since $u$ is continuous, it cannot change sign. Therefore either $u>0$ or $u<0$ on $\RN$.
	\end{proof}
	
	In order to examine the symmetry of the ground state solution, we introduce the useful tool of so-called {\it polarization}(see e.g.\cite{Weth}). Let $H\subset \RN$ be a closed half-space. We denote $\sigma_{H}$ the reflection with respect to the hyperplane $\partial H$. Then the polarization $u_{H}:\RN\to\R$ of a function $u:\RN\to \R$ with respect to $H$ is defined by 
	\begin{equation*}
		u_{H}(x):=\begin{cases}
			\max \{u(x), u(\sigma_{H}(x))\}\qquad\text{if }x\in H,\\
			\min \{u(x), u(\sigma_{H}(x))\}\qquad\text{if }x\notin H.
		\end{cases}
	\end{equation*}
	
	Recall the following two crucial lemmas from \cite[Lemma 5.3 and 5.4]{MVJFA}, which reflect a strong symmetric effect for nonlocal term. 
	\begin{lemma}\label{polarization inequality}
		Let $u\in L^{\frac{2N}{N+\a}}(\RN)$ and $H$ be a closed half-space in $\RN$. If $u\geq0$ and 
		\begin{equation*}
			\int_{\mathbb{R}^N} \int_{\mathbb{R}^N} \frac{u(x)u(y)}{|x - y|^{N-\alpha}} \, dx \, dy \geqslant \int_{\mathbb{R}^N} \int_{\mathbb{R}^N} \frac{u_H(x)u_H(y)}{|x - y|^{N-\alpha}} \, dx \, dy,
		\end{equation*}
		then either $u_H=u$ or $u_H=u\circ\sigma_{H}$.
	\end{lemma}
	
	\begin{lemma}\label{symmetrizaton}
		Let $ s \geqslant 1 $ and $ u \in L^s(\mathbb{R}^N)$. If $ u \geq 0$  and for every closed half-space $ H \subset \mathbb{R}^N$, 
		$u_H = u $ or  $u_H = u \circ \sigma_H$, then there exist $ x_0 \in \mathbb{R}^N$  and $ v : (0, \infty) \to \mathbb{R} $ a nonincreasing function such that for almost every $ x \in \mathbb{R}^N$, $u(x) = v(|x - x_0|)$.
	\end{lemma}
	
	\begin{proposition}\label{symmetry of ground state}
		Let $p>\frac{N+\a}{N}$. If $u\in \cM_1$ is a ground state solution to \eqref{eq}, then there exist $x_0\in\RN$ and $ v : (0, \infty) \to \mathbb{R} $ a monotone function such that for almost every $ x \in \mathbb{R}^N$, $u(x) = v(|x - x_0|)$ and $v(r)\to 0$ as $r\to+\infty$.
	\end{proposition}
	\begin{proof}
		Let $u\in \cM_1$ be a ground state solution to \eqref{eq} such that $\cI(u)=c_1$ and $H$ be any closed half-space. From Proposition \ref{constant sign} $u$ has constant sign. If $u<0$, it is clear that $-u>0$ is also a ground state solution to \eqref{eq}. Therefore, we need only to consider the case of positive solution. Recall that if $u$ is a weak solution to \eqref{eq}, by \cite[Proposition 2.10]{BIMM} we know that \( u \in W^{2,q}_{\mathrm{loc}}(\mathbb{R}^N) \) for each \( q \in [2, \infty) \), and \( |\nabla u| < 1 \) on \( \mathbb{R}^N \). Therefore, it is sufficient to exploit \cite[Lemma 3.1 (ii)]{Weth} for $\Psi_{0}(u)$ to get that 
		\begin{equation*}
			\Psi_{0}(u_{H})=\Psi_{0}(u).
		\end{equation*}
		In addition, there holds that
		\begin{equation*}
			\lVert u_{H}\rVert_{2}^2=\lVert u\rVert_{2}^2.
		\end{equation*}
		By Proposition \ref{embedding} \ref{embedding of Lp}, one has $|u|^{p}\in L^{\frac{2N}{N+\a}}(\RN)$. The polarization inequality for the Riesz kernel (see e.g., \cite{VanWill04}) gives
		\begin{equation*}
			\int_{\mathbb{R}^N} \int_{\mathbb{R}^N} \frac{|u_H(x)|^{p}|u_H(y)|^{p}}{|x - y|^{N-\alpha}} \, dx \, dy
			\geqslant
			\int_{\mathbb{R}^N} \int_{\mathbb{R}^N} \frac{|u(x)|^{p}|u(y)|^{p}}{|x - y|^{N-\alpha}} \, dx \, dy .
		\end{equation*}
		If the above inequality were strict, then $\cP_1(u_H)<0$. Thus, by Lemma \ref{uniqueness of parametre on Pohozaev} there would exist $t_H<1$ such that $(u_H)_{t_H}\in \cM_1$. Since $\Psi_0(u_H)=\Psi_0(u)$ and $\lVert u_H\rVert_2=\lVert u\rVert_2$, we would get
		\[
		c_1\leq \cI((u_H)_{t_H})
		<\frac{\a+2p}{N+\a+2p}\Psi_{0}(u)+\frac{\a+2p-2}{2(N+\a+2p)}\lVert u\rVert_{2}^2
		=\cI(u)=c_1,
		\]
		a contradiction. Hence equality holds, namely
		\begin{equation*}
			\int_{\mathbb{R}^N} \int_{\mathbb{R}^N} \frac{|u(x)|^{p}|u(y)|^{p}}{|x - y|^{N-\alpha}} \, dx \, dy = \int_{\mathbb{R}^N} \int_{\mathbb{R}^N} \frac{|u_H(x)|^{p}|u_H(y)|^{p}}{|x - y|^{N-\alpha}} \, dx \, dy,
		\end{equation*}
		Thus, it follows from Lemma \ref{polarization inequality} applied to $u^p$ that either $(u^p)_H=u^p$ or $(u^p)_H=u^p\circ\sigma_{H}$. Since $u>0$, this is equivalent to saying that either $u_H=u$ or $u_H=u\circ\sigma_{H}$. Since $H$ is arbitrary, the conclusion is derived from Lemma \ref{symmetrizaton}.
	\end{proof}
	\begin{proof}[Proof of Theorem \ref{thm of qualitative}]
		It follows from Proposition \ref{constant sign} and \ref{symmetry of ground state}.
	\end{proof}
	
	\section{Zero mass case and Sobolev-type inequality}\label{sec zero mass}
	Similar to Theorem \ref{Thm 1}, we first sketch the proof of the existence of minimizer for $c_0$. Throughout this section, the functional space is set by $X_0$, as well as $D(\Psi_{0})$ in the same meaning. Using the same notation in \eqref{scaling}, we define that 
	\begin{equation}\label{fiber function0mass}
		\begin{aligned}
			\cI_0(u_{t})&= t^N\int_{\mathbb{R}^N} 1 - \sqrt{1 - |\nabla u|^2}\, dx-\frac{t^{N+\a+2p}}{2p}\cF_{\a,p}(u),\\
			\frac{d\cI_0(u_{t})}{dt}&=t^{N-1}\left(N\int_{\mathbb{R}^N} 1 - \sqrt{1 - |\nabla u|^2}\, dx-\frac{N+\a+2p}{2p}t^{\a+2p}\cF_{\a,p}(u)\right),
		\end{aligned}
	\end{equation}
	and have an observation
	\begin{equation*}
		\frac{d\cI_0(u_{t})}{dt}\Big\rvert_{t=1}=\cP_0(u)\qquad\text{for }u\in D(\Psi_0)\backslash\{0\}.
	\end{equation*}
	It holds that
	\begin{lemma}\label{uniqueness of parametre on Pohozaev0mass}
		For each $u\in D(\Psi_0)\backslash\{0\}$, there exists unique $\tilde{t}_{u}\in (0,+\infty)$, which is a maximum of function $t\mapsto\cI_0(u_t)$, such that $\cP_0(u_{\tilde{t}_{u}})=0$, i.e., $u_{\tilde{t}_{u}}\in \cM_0$. Therefore, it holds that $\cP_0(u_t)>0$ for $0<t<\tilde{t}_u$ and $\cP_0(u_t)<0$ for $t>\tilde{t}_u$. Moreover, the map $u\mapsto \tilde{t}_{u}$ is continuous on $D(\Psi_0)\backslash\{0\}$.
	\end{lemma}
	\begin{proof}
		The proof is analogous to Lemma \ref{uniqueness of parametre on Pohozaev}.
	\end{proof}
	\begin{lemma}\label{minimizer=critical point0mass}
		Let $u_0\in \cM_0$ satisfy $\cI_0(u_0)=c_0$. Then $u_0$ is a critical point of $\cI_0$.
	\end{lemma}
	\begin{proof}
		We apply Lemma \ref{abstract critical point}. Let $X=X_0$,
		\[
		\mathfrak B(u):=\frac{1}{2p}\cF_{\a,p}(u),
		\qquad
		\mathfrak A(u):=\Psi_0(u),
		\]
		and let $S_tu=u_t$ be the scaling in \eqref{scaling}. Then
		$\cI_0=\mathfrak A-\mathfrak B$, $\mathfrak B\in\cC^1(X_0,\R)$ by Proposition
		\ref{differentiable of nonlocal}, and $\mathfrak A$ is proper, convex and
		lower semicontinuous on $X_0$, with $D(\mathfrak A)=D(\Psi_0)$ convex.
		Moreover, $S_1u=u$, $S_tu\nequiv0$ for $u\nequiv0$, $S_t(S_ru)=S_{tr}u$, and
		\[
		\mathfrak A(S_tu)=\Psi_0(u_t)=t^N\Psi_0(u)=t^N\mathfrak A(u).
		\]
		Thus {\rm (S1)--(S3)} hold with $\alpha(t)=t^N$. The fiber maps are of
		class $\cC^2$ by \eqref{fiber function0mass}.
		
		In this case,
		\[
		\mathfrak R(t,u):=\frac{d}{dt}\mathfrak B(u_t)
		=\frac{N+\a+2p}{2p}t^{N+\a+2p-1}\cF_{\a,p}(u),
		\]
		so $\mathfrak R\in\cC^1((0,+\infty)\times X_0,\R)$.  Therefore
		\[
		\mathfrak P(u)=\left.\frac{d}{dt}\cI_0(u_t)\right|_{t=1}
		=N\mathfrak A(u)-\mathfrak R(1,u)=\cP_0(u),
		\qquad
		\mathfrak M=\cM_0.
		\]
		Thus {\rm (A1)} and {\rm (A2)} follow directly from Lemma
		\ref{uniqueness of parametre on Pohozaev0mass}. Finally,
		\[
		\left.\frac{d}{dt}\left(\frac{\mathfrak R(t,u_0)}{\alpha'(t)}\right)\right|_{t=1}
		=\frac{(N+\a+2p)(\a+2p)}{2Np}\cF_{\a,p}(u_0)>0,
		\]
		so {\rm (A3)} is satisfied. Hence all assumptions of Lemma
		\ref{abstract critical point} hold, and $u_0$ is a critical point of
		$\cI_0$.
	\end{proof}
	
	\begin{lemma}\label{positive energy level0mass}
		Let $m_0$ be the mountain pass level defined by
		\begin{equation}\label{mountain pass level0mass}
			m_0:=\inf_{\gamma\in \Gamma} \max_{\tau \in [0,1]} \cI_0(\gamma(\tau)),\qquad  \Gamma := \{\gamma\in \C\left([0,1]; X_0\right): \gamma(0)=0,\,  \cI_0(\gamma(1))<0\}.
		\end{equation}
		Then it holds that $c_0= m_0>0$.
	\end{lemma}
	\begin{proof}
		From \eqref{fiber function0mass} one has that $\cI_0(u_t)\to -\infty$ as $t\to +\infty$ for any nonzero $u\in D(\Psi_{0})$, so that the definition for $m_0$ is well-defined.
		
		Using Hardy-Littlewood-Sobolev inequality and Proposition \ref{embedding} \ref{embedding of Lp}, one has, with the notation of $\frac{2Np}{N+\a}>2^*$, that
		\begin{equation*}
			\cF_{\a,p}(u)\leq  C \left(\irn |u|^{\frac{2Np}{N+\a}}\right)^{\frac{N+\a}{N}}\leq C\lVert u\rVert_{X_0}^{2p}\qquad \text{for } u\in X_0.
		\end{equation*}
		Therefore, we exploit \eqref{equivalent of Psi0}  to get 
		\begin{equation*}
			\begin{aligned}
				\cI_0(u) &= \Psi_0(u)- \frac{1}{2p}\cF_{\a,p}(u)\\
				&\geq \frac{1}{4}\left(\lVert \n u\rVert_{2}^2 + \lVert \n u\rVert_{r}^r\right)  - C\lVert u\rVert_{X_0}^{2p}\\
				& \geq \frac{1}{4}\lVert u\rVert_{X_0}^{2} - C\lVert u\rVert_{X_0}^{2p}
			\end{aligned}
		\end{equation*}
		for $u\in D(\Psi_{0})\backslash\{0\}$. Thus, we infer that there holds $\inf_{\lVert u\rVert_{X_0}=\beta}\cI_0(u)>0$ for some $\beta>0$, which implies that $m_0>0$.
		
		The proof for the identity of $c_0$ with $m_0$ is same as that of Lemma \ref{positive energy level} and we omit the detail here.
	\end{proof}
	
	We divide the proof of Theorem \ref{thm 0mass&sobolev ineq} into two parts, the existence of minimizer  and the Sobolev-type inequality, respectively.
	\begin{proof}[Proof of first part of Theorem \ref{thm 0mass&sobolev ineq}]
		Let $(u_n)_{n\in\N}\subset \cM_0$ be a minimizing sequence for $c_0$. Then we get from Lemma \ref{positive energy level0mass} and $\cP_0(u_n)=0$ that
		\begin{equation*}
			\cI_0(u_n) = \frac{\a+2p}{N+\a+2p}\Psi_{0}(u_n)\to c_0>0,
		\end{equation*}
		which implies that $\{\Psi_{0}(u_n)\}_{n\in \N}$ is bounded. Moreover, from \eqref{equivalent of Psi0} and $|\n u_n|\leq1$ a.e on $\RN$ for every $n\in \N$ we can infer that $(u_n)_{n\in\N}$ is bounded in $X_0$. By Lemma \ref{profile decomposition}, there is $k\in \N\cup \{+\infty\}$, $(y_{n}^{i})_{n\in\N}$ and $\tilde{u}_{i}\in X_0$ such that 
		\begin{equation}\label{splliting formulae0mass}
			\begin{aligned}
				u_{n}(\cdot+y_{n}^{i})&\weakto \tilde{u}_{i}\qquad\text{{\rm weakly in }}X_0,\, \text{{\rm as }}n\to\infty; \\
				u_{n}(x+y_{n}^{i})&\to \tilde{u}_{i}(x)\qquad\text{{\rm a.e. on }}\RN,\, \text{{\rm as }}n\to\infty;\\ 
				\lim_{n\to \infty}|y_{n}^{i}-y_{n}^{j}|&= +\infty\qquad\text{{\rm for }}i\neq j;\\
				\liminf_{n\to\infty}\Psi_0(u_n)&\geq \sum_{i=0}^{k}\Psi_{0}(\tilde{u}_{i});\\
				\lim_{n\to \infty}\cF_{\a,p}(u_n) &= \sum_{i=0}^{k}\cF_{\a,p}(\tilde{u}_{i}),
			\end{aligned}
		\end{equation}
		If $k=0$ and $\tilde{u}_0=0$, then 
		\begin{equation*}
			\cI_0(u_n) = \frac{\a+2p}{N+\a+2p}\Psi_{0}(u_n)\to 0
		\end{equation*}
		and $c_0=0$, which gives a contradiction. Thus, either $k>0$ or $k=0$ with $\tilde{u}_0\neq0$. Since $D(\Psi_{0})$ is convex and closed, then it is weakly closed and all $\tilde{u}_{i}$ are located in $D(\Psi_{0})$. Set $J:=\{i: \tilde{u}_i\nequiv 0\}\neq\emptyset$. By \eqref{splliting formulae0mass} and $\cP_0(u_n)=0$ we get
		\begin{equation*}
			\sum_{i\in J}N\Psi_{0}(\tilde{u}_i)\leq \liminf_{n\to \infty} N\Psi_{0}(u_n) = \liminf_{n\to \infty} \frac{N+\a+2p}{2p} \cF_{\a,p}(u_n)= \frac{N+\a+2p}{2p}\sum_{i\in J}\cF_{\a,p}(\tilde{u}_i).
		\end{equation*}
		This suggests us that there is $j\in J$ such that $\cP_0(\tilde{u}_j)\leq 0$, as well the parameter $t_j$ such that $t_{j}\tilde{u}_{j}(\frac{\cdot}{t_{j}}):= \tilde{w}_j\in \cM_0$ satisfies $t_j\leq 1$. On the other hand, by \eqref{splliting formulae0mass} and  Fatou's Lemma one has
		\begin{equation}\label{energy obligatory0mass}
			\begin{aligned}
				c_0\leq \cI_0(\tilde{w}_j)&=\frac{\a+2p}{N+\a+2p}\Psi_{0}(\tilde{w}_j)=\frac{\a+2p}{N+\a+2p}t_{j}^{N}\Psi_{0}(\tilde{u}_j)\\
				&\leq \frac{\a+2p}{N+\a+2p}\Psi_{0}(\tilde{u}_j)\leq \liminf_{n\to \infty}\frac{\a+2p}{N+\a+2p}\Psi_{0}(u_{n}(\cdot+y_{n}^{j}))=\liminf_{n\to \infty}\cI_0(u_n)= c_0.
			\end{aligned}
		\end{equation}
		Therefore, we conclude that $t_j=1$ and $\tilde{u}_j\in \cM_0$ with $\cI_0(\tilde{u}_j)=c_0$. Thus, $\tilde{u}_j$ is a minimizer for $c_0$. Moreover, combining \eqref{splliting formulae0mass} and \eqref{energy obligatory0mass} yields that $k=0$ and we relabel $\tilde{u}_j$ by $u_0$. Therefore, $u_0$ is a critical point of $\cI_0$ by Lemma \ref{minimizer=critical point0mass}. Then a similar discussion from the proof of Theorem \ref{Thm 1} is exploited to deduce that $u_0$ is a classical ground state solution to \eqref{eq} when $\omega=0$.
	\end{proof}
	
	\begin{proof}[Proof of second part of Theorem \ref{thm 0mass&sobolev ineq}]
		Let $u_0$ be the minimizer over $\cM_0$, that is, $\cI_0(u_0)=c_0$. For the assertion (i), let $u$ be a function in $X_0$ and $v=\tilde{t}_{u}u\left(\cdot/\tilde{t}_{u}\right)\in \cM_0$ by Lemma \ref{uniqueness of parametre on Pohozaev0mass}. Therefore, using $\cP_0(v)=0$ we have
		\begin{equation}\label{comparison bt c0  any}
			c_0=\cI_0(u_0) \leq \cI_0(v)= \frac{\a+2p}{N+\a+2p}\Psi_{0}(v) = \tilde{t}_{u}^{N}\frac{\a+2p}{N+\a+2p}\Psi_{0}(u).
		\end{equation}
		Moreover, $\cP_0(v)= 0$ yields that 
		\begin{equation}\label{resolve of tu}
			N\tilde{t}_{u}^{N}\Psi_{0}(u)= \frac{N+\a+2p}{2p}\tilde{t}_{u}^{N+\a+2p}\cF_{\a,p}(u).
		\end{equation}
		Consequently, we can infer \eqref{Sobolev-type inequality} by resolving $\tilde{t}_u$ from \eqref{resolve of tu} and plugging it into \eqref{comparison bt c0  any}.
		
		Furthermore, the equality holds if and only if $v=\tilde{t}_{u}u\left(\cdot/\tilde{t}_{u}\right)\in \cM_0$ is a minimizer for $c_0$ and the assertion of (ii) holds. Then we complete the entire proof.
	\end{proof}
	
	\begin{proof}[Proof of Theorem \ref{thm:zeroradial}]
		The argument is same as the proof of Theorem \ref{thm of qualitative}.
	\end{proof}

	{\bf Acknowledgements.}
	The authors were partly supported by the National Science Centre, Poland (Grant No. 2023/51/B/ST1/00968).
	
	
	\appendix
	
	\section{Minimizers are critical points -- abstract result}\label{Appendix}
	
	\begin{lemma}\label{abstract critical point}
		Let $X$ be a Banach space and let
		\[
		\mathfrak I=\mathfrak A-\mathfrak B,
		\]
		where $\mathfrak A:X\to(-\infty,+\infty]$ is proper, convex and lower semicontinuous, and
		$\mathfrak B\in\mathcal{C}^1(X,\mathbb{R})$. Set
		\[
		D(\mathfrak A):=\{u\in X:\mathfrak A(u)<+\infty\}
		\]
		and assume that $D(\mathfrak A)$ is convex.
		
		Let $S:(0,+\infty)\times X\to X$ be a scaling map, denoted by
		\[
		S_tu:=S(t,u),
		\]
		such that 
		\begin{enumerate}[label={\rm (S\arabic*)}]
			\item $S_1u=u$ and $S_tu\neq0$ whenever $u\neq0$;
			\item for any $t,r>0$ and $u\in X$, it holds $S_{tr}u=S_{t}(S_{r}u)$;
			\item it holds that $\mathfrak A(S_t u) = \alpha(t)\mathfrak A(u)$ for every $t>0$ and $u\in D(\mathfrak A)\setminus\{0\}$, where $\alpha\in \cC^2(0,+\infty)$ satisfies $\alpha(1)=1$ and $\alpha'(t)>0$ for all $t>0$.
		\end{enumerate}
		Suppose that, for every
		$u\in D(\mathfrak A)\setminus\{0\}$, the fiber map
		\[
		\gamma_u(t):=\mathfrak I(S_tu)
		\]
		is of class $\cC^1$ with respect to $t>0$. Moreover, assume that 
		$\mathfrak R :(0,+\infty)\times X\to \R$ given by
		\begin{equation*}
			\mathfrak R(t,u):=\frac{d}{dt}\mathfrak B(S_t u),\quad\hbox{for }(t,u)\in (0,+\infty)\times X
		\end{equation*}
		is of class $\cC^1$.
		Define the Pohožaev functional by
		\[
		\mathfrak P(u):=\left.\frac{d}{dt}\gamma_u(t)\right|_{t=1}
		=\alpha'(1)\mathfrak A(u)-\mathfrak R(1,u),
		\]
		and the Pohožaev set by
		\[
		\mathfrak M:=
		\{u\in D(\mathfrak A)\setminus\{0\}:\mathfrak P(u)=0\}.
		\]
		Let $u_0\in\mathfrak M$ satisfy
		\[
		\mathfrak I(u_0)=\inf_{\mathfrak M}\mathfrak I.
		\]
		Assume that the following conditions hold:
		\begin{enumerate}[label={\rm (A\arabic*)}]
			\item For every $u\in D(\mathfrak A)\setminus\{0\}$, there exists a unique $t_u>0$ such that
			\[
			S_{t_u}u\in\mathfrak M.
			\]
			Moreover, $t_u$ is the unique maximum point of the map
			\[
			t\mapsto \mathfrak I(S_tu).
			\]
			
			\item The map
			\[
			u\mapsto t_u
			\]
			is continuous on $D(\mathfrak A)\setminus\{0\}$.
			
			\item There holds      
			\[
			\left.\frac{d}{dt}\left(\frac{\mathfrak R(t,u_0)}{\alpha'(t)}\right)\right|_{t=1}>0.
			\]
		\end{enumerate}
		
		Then $u_0$ is a critical point of $\mathfrak I$ in the sense of Szulkin, namely
		\[
		\mathfrak A(u)-\mathfrak A(u_0)-\mathfrak B'(u_0)(u-u_0)\geq0
		\qquad\text{for all }u\in X.
		\]
	\end{lemma}
	
	\begin{proof}
		Let $u\in X$. If $u\notin D(\mathfrak A)$, then $\mathfrak A(u)=+\infty$ and the conclusion is satisfied. Hence we may assume that $u\in D(\mathfrak A)$.
		
		Let $\tau_n\to0^+$ and set
		\[
		v_n:=(1-\tau_n)u_0+\tau_n u.
		\]
		By the convexity of $D(\mathfrak A)$, we have $v_n\in D(\mathfrak A)$. We may also assume that \(v_n\neq0\). By {\rm (A1)}, let
		\[
		t_n:=t_{v_n},
		\qquad
		w_n:=S_{t_n}v_n\in\mathfrak M.
		\]
		Since $u_0$ minimizes $\mathfrak I$ on $\mathfrak M$, by (A1) we obtain
		\[
		0\leq \mathfrak I(w_n)-\mathfrak I(u_0)=\mathfrak I(w_n)-\mathfrak I(v_n)
		+
		\mathfrak I(v_n)-\mathfrak I(u_0).
		\]
		Therefore, by the convexity of $\mathfrak A$ and the differentiability of $\mathfrak B$, we get
		\begin{equation}\label{minimization equality}
			\begin{aligned}
				\mathfrak I(v_n)-\mathfrak I(w_n) \leq \mathfrak I(v_n)-\mathfrak I(u_0) &=
				\mathfrak A(v_n)-\mathfrak A(u_0)-\mathfrak B(v_n)+\mathfrak B(u_0)\\
				&\leq
				\tau_n\big(\mathfrak A(u)-\mathfrak A(u_0)\big)
				-\tau_n\mathfrak B'(u_0)(u-u_0)+o(\tau_n).
			\end{aligned}
		\end{equation}
		On the other hand, from the definition of fiber map $\gamma_u$ and mean value theorem, one has 
		\begin{equation*}
			\mathfrak I(v_n)-\mathfrak I(w_n) = \int_{t_n}^{1}\frac{d}{dt} \gamma_{v_n}(t)\, dt = (1-t_n)\left.\frac{d}{dt}\gamma_{v_n}(t)\right|_{t=\xi_n},
		\end{equation*}
		where $\xi_n$ lies between $t_n$ and 1. Therefore, dividing \eqref{minimization equality} by $\tau_n$, we obtain
		\begin{equation}\label{quatio inequality}
			\frac{1- t_n}{\tau_n}\left.\frac{d}{dt}\gamma_{v_n}(t)\right|_{t=\xi_n}\leq \mathfrak A(u)-\mathfrak A(u_0)-\mathfrak B'(u_0)(u-u_0) +o(1).
		\end{equation}
		From assumption {\rm (A2)}, it is clear that $t_n\to 1 $ as $n\to \infty$. Since $\xi_n$ lies between $t_n$ and $1$, we also have $\xi_n\to1$. By the lower semicontinuity of $\mathfrak A$,
		\[
		\mathfrak A(u_0)\leq \liminf_{n\to\infty}\mathfrak A(v_n).
		\]
		On the other hand, the convexity of $\mathfrak A$ gives
		\[
		\limsup_{n\to\infty}\mathfrak A(v_n)\leq
		\limsup_{n\to\infty}\big((1-\tau_n)\mathfrak A(u_0)+\tau_n\mathfrak A(u)\big)=\mathfrak A(u_0).
		\]
		Thus $\mathfrak A(v_n)\to\mathfrak A(u_0)$. Together with the continuity of $\mathfrak R$, this gives
		\begin{equation}\label{eq:appendgamma}
			\lim_{n\to \infty}\left.\frac{d}{dt}\gamma_{v_n}(t)\right|_{t=\xi_n}=\lim_{n\to \infty} (\alpha'(\xi_n)\mathfrak A(v_n)-\mathfrak R(\xi_n,v_n))=\mathfrak P(u_0)=0.
		\end{equation}
		We first show that
		\begin{equation}\label{upper bound for quotient}
			\limsup_{n\to\infty}\frac{1-t_n}{\tau_n}<+\infty.
		\end{equation}
		Indeed, suppose by contradiction that, up to a subsequence,
		\[
		\frac{1-t_n}{\tau_n}\to+\infty.
		\]
		Then $t_n<1$ for all large $n$ and $\tau_n/(1-t_n)\to0$. Since
		$w_n=S_{t_n}v_n\in\mathfrak M$, by (A1), 
		\begin{equation}\label{eq:appengamma}
			\frac{d}{dt}\gamma_{v_n}(t)|_{t=t_n}=0,
		\end{equation} and we obtain
		\[
		\mathfrak I(w_n)=\mathfrak E(t_n,v_n),
		\qquad
		\mathfrak I(u_0)=\mathfrak E(1,u_0),
		\]
		where
		\[
		\mathfrak E(t,u):=
		\frac{\alpha(t)}{\alpha'(t)}\mathfrak R(t,u)-\mathfrak B(S_tu).
		\]
		Notice that, under the preceding assumptions, the map $\mathfrak{E}$ is of class \(\mathcal C^1\) on \((0,+\infty)\times X\). Moreover
		using the minimality of $u_0$ on $\mathfrak M$, and $v_n-u_0=\tau_n(u-u_0)$, we obtain
		\begin{equation*}
			\begin{aligned}
				0&\leq \mathfrak I(w_n)-\mathfrak I(u_0)
				=\mathfrak E(t_n,v_n)-\mathfrak E(1,u_0)\\
				&=(t_n-1)\left.\frac{d}{dt}\mathfrak E(t,u_0)\right|_{t=1}
				+o(|t_n-1|)+O(\tau_n).
			\end{aligned}
		\end{equation*}
		Moreover,
		\[
		\left.\frac{d}{dt}\mathfrak E(t,u_0)\right|_{t=1}
		=\alpha(1)\left.\frac{d}{dt}\left(\frac{\mathfrak R(t,u_0)}{\alpha'(t)}\right)\right|_{t=1}>0
		\]
		by {\rm (A3)}. Dividing the preceding inequality by $1-t_n>0$ and
		letting $n\to\infty$, we get a contradiction. Hence
		\eqref{upper bound for quotient} holds. Therefore, either $\{(1-t_n)/\tau_n\}_{n\in\N}$ is bounded or
		\begin{equation}\label{explosion}
			\liminf_{n\to \infty}\frac{1-t_n}{\tau_n}= -\infty.
		\end{equation}
		If the first case holds, then the conclusion is obtained from \eqref{quatio inequality} and \eqref{eq:appendgamma}. Suppose that \eqref{explosion} holds. Without loss of generality, we assume that $t_n>1$ for all $n\in\N$. Set
		\[
		G(t,v):=\frac{\mathfrak R(t,v)}{\alpha'(t)},\quad\hbox{for }t>0,\;v\in X.
		\]
		In view of \eqref{eq:appengamma}, we have
		\[
		0=\left.\frac{d}{dt}\mathfrak I(S_tv_n)\right|_{t=t_n}
		=\alpha'(t_n)\mathfrak A(v_n)-\mathfrak R(t_n,v_n),
		\]
		and hence
		\[
		\mathfrak A(v_n)=G(t_n,v_n).
		\]
		Assume first that $\mathfrak A(v_n)\leq \mathfrak A(u_0)$ for all $n\in\mathbb{N}$. Then by $u_0\in \mathfrak M$ and the preceding identity, we get from definition of Poho\v{z}aev set that
		\begin{equation*}
			G(1,u_0)=\frac{\mathfrak R(1,u_0)}{\alpha'(1)} = \mathfrak A(u_0) \geq \mathfrak A(v_n) = \frac{\mathfrak R(t_n,v_n)}{\alpha'(t_n)}
			=G(t_n,v_n).
		\end{equation*}
		Noticing that $t_n\to1$, $v_n\to u_0$ in $X$, and $\tau_n/(t_n-1)\to0$ along this subsequence, we obtain
		\begin{equation*}
			\begin{aligned}
				0&\geq G(t_n,v_n)-G(1,u_0)\\
				&=G(t_n,u_0)-G(1,u_0)+G(t_n,v_n)-G(t_n,u_0)\\
				&=(t_n-1)\left.\frac{d}{dt}G(t,u_0)\right|_{t=1}+o(t_n-1)+O(\tau_n).
			\end{aligned}
		\end{equation*}
		Dividing by $t_n-1>0$ and passing to the limit, we obtain
		\[
		0\geq \left.\frac{d}{dt}G(t,u_0)\right|_{t=1},
		\]
		which contradicts {\rm (A3)}. Thus, up to a subsequence, we may assume that $\mathfrak A(v_n)>\mathfrak A(u_0)$. In this case, by convexity of $\mathfrak A$, we have
		\begin{equation*}
			0\leq \mathfrak A(v_n)-\mathfrak A(u_0) \leq \tau_n \mathfrak A(u)+ (1-\tau_n)\mathfrak A(u_0) -\mathfrak A(u_0)=\tau_n\left(\mathfrak A(u)- \mathfrak A(u_0)\right).
		\end{equation*}
		Thus, we get the following estimate:
		\begin{equation*}
			\begin{aligned}
				0&\leq \frac{\mathfrak A(v_n)-\mathfrak A(u_0)}{\tau_n} \\
				&= \frac{G(t_n,v_n)-G(1,u_0)}{\tau_n}\\
				&= \frac{t_n-1}{\tau_n}\left(\left.\frac{d}{dt}G(t,u_0)\right|_{t=1}+o(1)\right)+O(1)\\
				& \leq \mathfrak A(u)- \mathfrak A(u_0)
			\end{aligned}
		\end{equation*}
		for all large $n$. By {\rm (A3)}, this implies
		\[
		\limsup_{n\to\infty} \frac{t_n-1}{\tau_n} < +\infty,
		\]
		which contradicts \eqref{explosion}. Thus, we complete the entire proof.
	\end{proof}

	\subsection{Application to the Nehari-Pankov constraint}
	
	We recall a model case related to the Nehari-type construction used by
	Pankov \cite{Pankov} and Szulkin, Weth \cite{SW}. Let us consider the {\em nonlinear Schr\"odinger equation}
	\begin{equation}\label{eq:Schr}
		-\Delta u+V(x)u=f(x,u)
		\qquad\text{in }\mathbb R^N.
	\end{equation}
	The {\em energy functional} associated with the problem 
	$ I:H^1(\R^N)\to\R$ is defined by
	$ I(u)=\frac12 Q(u)-\int_{\R^N}F(x,u)\,dx,$ 
	where the quadratic form is 
	$$
	Q(u):=
	\int_{\RN}\left(|\nabla u|^2+V(x)|u|^2\right)\,dx,\hbox{
		for }u\in H^1(\RN).$$
	We set
	$
	F(x,s):=\int_0^s f(x,\tau)\,d\tau
	$
	and assume the following assumptions.
	\begin{enumerate}[label={\rm (SW\arabic*)}]
		\item $V$ is continuous, $1$-periodic in $x_1,\ldots,x_N$, and
		\[
		0\notin \sigma(-\Delta+V),
		\]
		where $\sigma(-\Delta+V)$ denotes the spectrum of $-\Delta+V$.
		
		\item $f$ is continuous in $x$, and of $\cC^1$ class in $s$.  Moreover
		there exist $a>0$ and $p\in(2,2^*)$ such that
		\[
		|\partial_s f(x,s)|\leq C\left(1+|s|^{p-2}\right)
		\qquad
		\text{for all } (x,s)\in\mathbb R^N\times\mathbb R.
		\]
		
		\item
		$
		f(x,u)=o(u)$ as $|u|\to 0,
		$
		uniformly in $x\in\mathbb R^N$.
		
		\item
		$
		\frac{F(x,u)}{u^2}\to+\infty$ as $|u|\to+\infty,
		$
		uniformly in $x\in\mathbb R^N$.
		
		\item For every $x\in\mathbb R^N$, the map
		\[
		u\mapsto \frac{f(x,u)}{|u|}
		\]
		is strictly increasing on $(-\infty,0)$ and on $(0,+\infty)$.
	\end{enumerate}

	In view of {\rm (SW1)}, there are closed subspaces $E^+$ and $E^-$ such that $H^1(\R^N)=E^+\oplus E^-$ and $ Q$ is positive definite on $E^+$ and negative definite on $E^-$. Moreover
	$$ Q(u)=\|u^+\|^2-\|u^-\|^2\hbox{ for }u=u^++u^-\in E^+\oplus E^-.$$
	As in \cite{SW,Pankov}, $\mathfrak{I}$ is of class $\cC^1$ and we introduce the following Nehari-Pankov manifold
	\[
	\mathcal N
	:=
	\left\{
	u\in H^1(\mathbb R^N)\setminus\{0\}:
	\int_{\mathbb R^N}\bigl(|\nabla u|^2+V(x)u^2\bigr)\,dx
	=
	\int_{\mathbb R^N}f(x,u)u\,dx,\;\hbox{and } I'(u)|_{E^-}=0
	\right\}.
	\]
	We show the main result of the application of our abstract Lemma \ref{abstract critical point}.
	\begin{theorem}\label{ThNehariPankov}
		Suppose that {\rm (SW1)--(SW5)} are satisfied. 
		Then any $u_0\in\cN$ satisfying
		\[
		I(u_0)=\inf_{u\in\mathcal N} I(u)
		\]
		is a weak solution of \eqref{eq:Schr}.
	\end{theorem}
	
	We shall give the proof after recalling some preliminary properties.
	
	Let 
	$\mathfrak W:H^1(\R^N)\to\R$ be given by $$\mathfrak W(u):=\frac12 \|u^-\|^2+\int_{\R^N}F(x,u)\,dx,\hbox{ for }u\in H^1(\R^N).$$
	Observe that $\mathfrak W$ is continuous, strictly convex (see \cite[Lemma 2.2]{SW}) and coercive. Then, let
	$w(u)\in E^-$ be a unique minimizer of 
	$$E^-\ni w\mapsto \mathfrak W(u+w)\in\R.$$
	One can show that the map $u\mapsto \mathfrak W(u+w(u))$ is of class $\cC^1$, see e.g.  (i)--(iii) in the proof of \cite[Theorem 4.4]{BM_JFA}. Moreover, it follows that
	\begin{equation}\label{eq:w-characterization}
		\mathfrak{W}'(u+ w(u))[v] =  \langle w(u), v\rangle + \irn f(x,u+w(u))v\, dx =0 \qquad\hbox{for }v\in E^-.
	\end{equation}
	We set now,
	$X:=E^+$, $\mathfrak A(u)=\frac12 \|u^+\|^2$, $\mathfrak B(u):=\mathfrak W(u+w(u))$ and functional $\mathfrak{I}: X\to\R$ defined by $\mathfrak{I}:= \mathfrak{A}-\mathfrak{B}$. 
	\begin{lemma}\label{lem:w-C1} 
		The following statements hold:
		\begin{enumerate}[label={\rm (\roman*)}, ref={\rm (\roman*)}]
			\item $w:X\to E^-$ is of $\cC^1$-class,
			\item The functional $\mathfrak{B}: X\to \R$ is of $\cC^1$ class,
			\item $\mathfrak{B}'(u) = \left.\mathfrak{W}'(u+w(u))\right|_{X}: X\to \R$ for every $u\in X$.
		\end{enumerate}
	\end{lemma}
	\begin{proof}
		Define
		\[
		G:E^+\times E^-\to (E^-)',
		\qquad
		G(u,w)[v]
		:=
		\langle w,v\rangle
		+
		\int_{\mathbb R^N}f(x,u+w)v\,dx .
		\]
		Then \eqref{eq:w-characterization} is equivalent to
		\[
		G(u,w(u))=0.
		\]
		By the assumptions on \(f\), the map \(G\) is of class \(\cC^1\), and
		\[
		G''(u,w)[\eta][v]
		=
		\langle \eta,v\rangle
		+
		\int_{\mathbb R^N}f_s(x,u+w)\eta v\,dx,
		\qquad \eta,v\in E^-.
		\]
		Moreover, by {\rm (SW3)} and {\rm (SW5)}, one has \(f_s(x,s)\ge0\). Therefore
		\[
		G''(u,w(u))[\eta][\eta]
		=
		\|\eta\|^2
		+
		\int_{\mathbb R^N}f_s(x,u+w(u))\eta^2\,dx
		\ge
		\|\eta\|^2
		\qquad\forall \eta\in E^-.
		\]
		Thus \(G''(u,w(u)):E^-\to(E^-)'\) is an isomorphism by the
		Lax--Milgram theorem. Hence the implicit function theorem applies to
		\(G(u,w)=0\), and we obtain $w\in \cC^1(X,E^-)$, so that (i) holds. Finally, one can follow the proof of \cite[Theorem 4.4]{BM_JFA} to show (ii) and (iii).
	\end{proof}
	Let us define
	\[
	S_tu:=tu,\quad t>0.
	\]	
	Note that
	$
	\mathfrak A
	$
	is proper, convex, continuous, and therefore lower semicontinuous. Moreover,$D(\mathfrak A)=E^+$ is convex.
	For $t>0$ and $u\in X$, we have
	\[
	\mathfrak A(S_tu)
	=
	\mathfrak A(tu)
	=
	t^2\mathfrak A(u).
	\]
	Thus in Lemma \ref{abstract critical point} we take
	\[
	\alpha(t)=t^2,
	\qquad
	\alpha'(t)=2t,
	\qquad 
	\gamma_u(t) = \mathfrak{I}(S_t u)=t^2\mathfrak A(u)- \mathfrak{B}(S_t u)
	\]
	Clearly, $S_1u=u$, $S_tu\neq0$ whenever $u\neq0$, and
	\[
	S_{tr}u=tr\,u=S_t(S_ru)
	\qquad
	\text{for all }t,r>0.
	\]
	Therefore {\rm (S1)}--{\rm (S3)} in Lemma \ref{abstract critical point}
	are satisfied. Then from Lemma \ref{lem:w-C1} we have 
	\begin{equation*}
		\mathfrak R(t,u):= \frac{d}{dt} \mathfrak{B}(S_t u)=\mathfrak{W}'(tu+w(tu))[u] =\irn f(x,tu+w(tu))u\, dx,\qquad\hbox{ for }u\in X
	\end{equation*}
	Thus, as in Lemma \ref{abstract critical point}, we can define the following functional by
	\begin{equation*}
		\mathfrak{P}(u):= 2\mathfrak{A}(u) - \irn f(x,u+w(u))u\, dx,
	\end{equation*}
	and the following constraint by
	\begin{equation*}
		\mathfrak{M}:= \{u\in X\setminus\{0\}: \mathfrak{P}(u)=0\}.
	\end{equation*}
	\begin{lemma}\label{lem:homeomorphism}
		The map $m:\mathfrak M \to \cN$ defined by $u\mapsto u+ w(u)$ is a homeomorphism between $\mathfrak{M}$ and the Nehari-Pankov constraint $\mathcal{N}$. Moreover, $u_0=u_0^++u_0^-$ is a minimizer of $I$ on $\mathcal N$ if and only if $u_0^+\in\mathfrak M$ is a minimizer of $\mathfrak{J}$ on $\mathfrak M$.
	\end{lemma}
	\begin{proof}
		Indeed, for $u\in \mathfrak{M}$, there holds $\mathfrak{P}(u)=0$. Moreover, from \eqref{eq:w-characterization} we have 
		$$\mathfrak{W}'(u+w(u))[v]=I'(u+w(u))[v]=0$$ 
		for all $v\in E^-$. Plus $\mathfrak{P}(u)=0$ with \eqref{eq:w-characterization} yields that $I'(u+w(u))[u+w(u)]=0$. These imply that $u+w(u)\in \mathcal{N}$. Conversely, let \(z\in\mathcal N\) and write \(z=z^++z^-\). Since
		\(I'(z)|_{E^-}=0\), for every \(v\in E^-\) we have
		\[
		0=I'(z)[v]
		=
		-\langle z^-,v\rangle
		-
		\int_{\mathbb R^N}f(x,z)v\,dx.
		\]
		Equivalently,
		\[
		\mathfrak W'(z^+ + z^-)[v]=0
		\qquad\hbox{ for every }v\in E^-.
		\]
		Thus \(z^-\) is the unique minimizer of
		\[
		E^-\ni w\mapsto \mathfrak W(z^+ + w),
		\]
		and therefore $z^-=w(z^+)$.
		Hence \(z=z^+ + w(z^+)\). By \(z\in\mathcal N\), we also have
		\(I'(z)[z]=0\). Together with \(I'(z)[z^-]=0\), this gives
		\[
		0=I'(z)[z^+]
		=
		I'(z^+ + w(z^+))[z^+]
		=
		\mathfrak P(z^+).
		\]
		Therefore \(z^+\in\mathfrak M\).
		Moreover, since $I(u+w(u))= \mathfrak{I}(u)$, the correspondence of these two different minimizers immediately holds.
	\end{proof} 
	\begin{proof}[Proof of Theorem \ref{ThNehariPankov}]
		It remains to verify the hypotheses (A1)--(A3) of Lemma \ref{abstract critical point}. In view of \cite[Proposition 2.3]{SW}, we see that there exists a unique maximum point $t_u>0$ of the fiber map $\gamma_u(t):=\mathfrak I(S_tu)$ for every $u\in X$, so that (A1) holds. Moreover, $S_{t_u}u= t_u u\in \mathfrak{M}$. Next, one infer from \cite[Lemma 2.8]{SW} and Lemma \ref{lem:homeomorphism} that the map $u\mapsto t_u$ is continuous, which implies (A2). It remains to show (A3). Suppose $u_0=u_0^++u_0^-$ is a minimizer over $\mathcal N$ where $u_{0}^+\in X$ and $u_0^-= w(u_0^+)\in E^-$ from Lemma \ref{lem:homeomorphism}.  
		Differentiating \eqref{eq:w-characterization} at $ u_{0}^+ $ in the direction $ u_{0}^+ $, we get 
		\begin{equation}\label{eq:intersection}
			\langle w'( u_{0}^+ )[ u_{0}^+ ],v\rangle - \irn f_s(x, u_{0}^+ + w(u_{0}^+))(u_{0}^+ + Dw(u_{0}^+)[u_{0}^+ ])v\, dx=0
		\end{equation} 
		for all $v\in E^-$. Besides, we have
		\begin{equation}\label{eq:second derivative}
			\begin{aligned}
				\left.\frac{d^2}{dt^2}\gamma_{u_{0}^+}(t)\right|_{t=1} &= \mathfrak{A}(u_{0}^+) - \frac{d}{dt}\mathfrak{R}(1,u_{0}^{+})\\
				&= \lVert u_{0}^+\rVert^2 - \irn f_s (x, u_{0}^+ + w(u_{0}^+))(u_{0}^++Dw(u_{0}^+)[u_{0}^+])u_{0}^+.
			\end{aligned}
		\end{equation}
		Let $v= w'(u_{0}^+)[u_{0}^+ ]$ in \eqref{eq:intersection} and $q:= u_{0}^+ +v$. Then combining \eqref{eq:intersection} and \eqref{eq:second derivative} yields that 
		\begin{equation}\label{quadratic form}
			\left.\frac{d^2}{dt^2}\gamma_{u_{0}^+}(t)\right|_{t=1} = I''(u_0)[q,q].
		\end{equation} 
		We next show that $I''(u_0)$ is negative definite on $\R u_0\oplus E^-$. Let
		$h=\tau u_0+v$ with $\tau\in\R$ and $v\in E^-$. Since $u_0\in\cN$,
		\[
		I'(u_0)[u_0]=0,\qquad I'(u_0)[v]=0.
		\]
		Thus
		\[
		Q(u_0,u_0)=\irn f(x,u_0)u_0\,dx,\qquad
		Q(u_0,v)=\irn f(x,u_0)v\,dx,
		\]
		and $Q(v,v)=-\lVert v\rVert^2$. Therefore
		\[
		\begin{aligned}
			I''(u_0)[h,h]
			&=
			Q(h,h)-\irn f_s(x,u_0)h^2\,dx\\
			&=
			-\lVert v\rVert^2
			-\irn
			\left(f_s(x,u_0)-\frac{f(x,u_0)}{u_0}\right)
			(\tau u_0+v)^2\,dx
			-\irn \frac{f(x,u_0)}{u_0}v^2\,dx.
		\end{aligned}
		\]
		By {\rm (SW3)} and {\rm (SW5)}, we have $I''(u_0)[h,h]\leq 0$. Moreover, the strict
		differential monotonicity condition gives
		\[
		\irn
		\left(f_s(x,u_0)u_0^2-f(x,u_0)u_0\right)\,dx>0.
		\]
		Hence
		\[
		I''(u_0)[h,h]<0
		\qquad
		\hbox{for every }0\neq h\in \R u_0\oplus E^-.
		\]
		Applying this to \eqref{quadratic form} we obtain that \eqref{eq:second derivative} is strictly negative. Consequently,
		we arrive at 
		\begin{equation*}
			\begin{aligned}
				\frac{d}{dt}\left.\frac{\mathfrak{R}(t,u_{0}^+)}{\alpha'(t)}\right|_{t=1} &= \frac{d}{dt}\left.\frac{\irn f(x,tu_{0}^++w(tu_{0}^+))\, dx}{2t}\right|_{t=1} \\
				&= \frac 12\irn f_s(x,u_{0})(u_{0}^+ +Dw(u_{0}^+)[u_{0}^+])u_{0}^+\, dx -\frac 12 \irn f(x,u_{0}))u_{0}^+] \, dx\\
				& > \frac{1}{2} \lVert u_{0}^+\rVert^2 -\frac 12 \irn f(x,u_{0})u_{0}^+ \, dx\\
				&= \frac{1}{2}\left[\lVert u_0\rVert^2- \irn f(x,u_{0})u_{0}\, dx \right] \\
				&= I'(u_0)[u_0]=0,
			\end{aligned}
		\end{equation*}
		which is the required condition {\rm (A3)}. Thus, $u_0^+$ is a critical point of $\mathfrak I$ on $X= E^+$. Since $\mathfrak A$ is differentiable in
		this model case, this means
		\[
		I'(u_0^+)[v]=\mathfrak I'(u_0^+)[v]=0
		\qquad\hbox{for all }v\in E^+.
		\]
		By \eqref{eq:w-characterization}, we
		also have
		$\mathfrak I'(u_0)[\phi]=0$ for all $\phi\in E^-$.
		Combining the two relations and writing every $z\in H^1(\RN)$ as
		$z=z^++z^-$, we conclude that
		\[
		I'(u_0)[z]=\mathfrak I'(u_0)[z]=0\qquad\hbox{for all }z\in H^1(\RN).
		\]
		Thus $u_0$ is a weak solution of \eqref{eq:Schr}.
	\end{proof}

	The conclusion of Theorem \ref{ThNehariPankov} was proved in \cite{SW}
	under the additional assumption that \(f(\cdot,s)\) is \(1\)-periodic in
	\(x\), but under a weaker regularity hypothesis than our {\rm (SW2)}:
	namely, it was assumed only that \(f\) is continuous and satisfies
	\[
	|f(x,s)|\leq C\left(1+|s|^{p-1}\right).
	\]
	In that setting, Szulkin and Weth also established the existence of a
	ground state by means of Ekeland's variational principle. Our purpose here
	is different: we observe that the fact that any minimizer on the constraint
	is a critical point can be obtained by an alternative method based on Lemma
	\ref{abstract critical point}. The separate problem of proving the existence
	of such a minimizer may then be studied independently by other tools, for
	instance by concentration--compactness arguments. We also see potential for
	applying Lemma \ref{abstract critical point} to nonsmooth nonlinear
	Schr\"odinger equations; this will be addressed in future work.

\end{document}